\documentclass[14pt,a4paper]{article}
\usepackage{version}
\usepackage{float}
\usepackage{amsmath}
\usepackage{amstext}
\usepackage{tabularx}
\usepackage[integrals]{wasysym}
\usepackage{marvosym}
\usepackage{manfnt}
\usepackage{ifsym}
\usepackage[psamsfonts]{amssymb}
\usepackage[unicode,bookmarks,bookmarksopen,bookmarksopenlevel=0,colorlinks,%
plainpages=false,pdfpagelabels,hypertexnames=false,pageanchor=true,breaklinks=true,%
linkcolor=blue,citecolor=blue,urlcolor=red]{hyperref}
\usepackage[left=2.5cm, right=2.5cm, top=2.5cm, bottom=2.5cm,
includehead]{geometry}
\usepackage{upref}  %ссылки независимо от контекста печатаются прямыми %
\usepackage{indentfirst}
\newif\ifpdf
\ifx\pdfoutput\undefined
\pdffalse % we are not running PDFLaTeX
\else
\pdfoutput=1 % we are running PDFLaTeX
\pdftrue \fi
\ifpdf
\usepackage[pdftex]{graphicx}
\usepackage{graphicx}
\else
\usepackage{graphicx}
\fi

\newcommand{\wtlnbl}{{\widetilde{\nabla}}}
\newcommand\pr{\partial}
\newcommand\ben{\begin{enumerate}}
\newcommand\een{\end{enumerate}}
\newcommand\bed{\begin{itemize}}
\newcommand\eed{\end{itemize}}
\newcommand\bei{\begin{description}}
\newcommand\eei{\end{description}}
\newcommand\beear{\begin{eqnarray}}
\newcommand\eear{\end{eqnarray}}
\newcommand\beq{\begin{eqnarray*}&}
\newcommand\eeq{&\end{eqnarray*}}

\newcommand\wqx{\widetilde{q}^\xi}
\newcommand\wqta{\widetilde{q}^\tau}
\newcommand\bqx{\overline{q}^\xi}

\newcommand\pcr{{p}^\circ}
\newcommand\pbl{{p}^\bullet}
\newcommand\efr{\mathfrak{e}}
\newcommand\astphi{{\phi_\ast}}
\newtheorem{prpsn}{Proposition}
\newtheorem{lemma}{Lemma }%[section]
\newtheorem{theorem}{Theorem}%[section]
\newtheorem{dfntn}{Definition}%[section]
\newtheorem{rmrk}{Remark}

\newcommand\IM{\operatorname{Im} \,}

\newcommand\const{\operatorname{const}}
\newcommand\Div{\operatorname{div}}
\newcommand\sgn{\operatorname{sgn}}

\def\eps{{\epsilon}}

\newcommand{\byd}{\stackrel{\mathrm{def}}{=}}

\newcommand\vb{\boldsymbol{b}}
\newcommand\vc{\boldsymbol{c}}

\newcommand\vu{\boldsymbol{u}}
\newcommand\vv{\boldsymbol{v}}
\newcommand\vw{\boldsymbol{w}}
\newcommand\vx{\boldsymbol{x}}
\newcommand\vy{\boldsymbol{y}}

\newcommand\vg{\boldsymbol{g}}

\newcommand\vomega{\boldsymbol{\omega}}

\newcommand{\hx}{{\hat{x}}}
\usepackage{fancyhdr}
\usepackage{lastpage}
\begin{document}
\addtocounter{page}{0}
\thispagestyle{empty}
\begin{center}
\textbf{Short-wave signal versus indirect prey-taxis}
\\
Running title: Short waves VS indirect prey-taxis
\\
Authors:\\
\textbf{Andrey Morgulis, the corresponding author,\\ E-mail: abmorgulis@sfedu.ru, morgulisandrey@gmail.com,\\ ORCID 0000-0001-8575-4917X}\\
 I.I. Vorovich Inst. Math. Mech. \& Comp. Sci.,
Southern Federal University, Rostov-on-Don, Russia,\\ 
and \\
Southern Mathematical Institute VSC RAS, Vladikavkaz, Russia.\\
\textbf{Karrar Malal,}\\
I.I. Vorovich Inst. Math. Mech. \& Comp. Sci.,
Southern Federal University, Rostov-on-Don, Russia,
\end{center}

\begin{abstract}\noindent
{\small We address short-wave asymptotics for one class of quasi-linear second-order  PDE's systems involving  the cross-diffusion described by the so-called  Patlak-Keller--Segel law. It is common to  employ these equations for modeling   the predator--prey community with the prey--taxis, which  is the interaction of two species of particles or cells  through which the predator species is capable of moving directionally while searching for the prey species. However, we  suppose the predators to be sensitive not to the prey density but to another cue field produced  by the prey. Additionally, producing the cue is assumed to be sensitive to the intensity of an external field, which is independent from the community state and mimics the environmental inhomogeneity due to natural or artificial reasons. This is what we call the "external signal".  We assume that the  external signal takes a general short-wave form and construct a  formal asymptotic expansion for the    short-wave solutions with no restrictions on the spatial dimension or kinetics of inter- or intra-specific reactions modulo solving a hierarchy of the so-called slow systems. With no special assumption, we  explicitly write down both the  leading term in the asymptotics and the leading slow system, which governs the slow evolution of the amplitude of the leading term.  Further, we apply  the leading approximation to study the stability or instability induced  by the external signal  following Kapitza's theory for the upside-down pendulum. Applying the general results to some special classes of external signals, we construct examples of suppressing the taxical transport, examples of the robustness of the species equilibrium to the signal up to a very strong stabilization, and, oppositely, destabilization and blurring the borderline in the parametric space between the areas of stability and instability of this equilibrium.}
\end{abstract}

\section{Introduction and problem setting}\label{ScIntr}
\noindent
This study examines the interactions of two different species, conditionally speaking, predators and prey, each composed of indistinguishable particles. Diffusion, as well as the ability of predators to generate a macroscopic mass flow while searching for prey, together determine the spatial distribution of species. The latter phenomenon is called ``prey--taxis," following article \cite{KrvOdll}, which arguably pioneered its study. A widely accepted (though not the only) model for taxis is the Patlak--Keller--Segel (PKS) model. It posits that tactic flow is parallel to the intensity gradient of a cue field that guides predators in their search for prey. Often (though not always), the cue is the prey concentration. The prey-taxis using a different cue is called indirect. In either case, adding the  equation governing the cue field (if necessary) results in a system of second-order partial differential equations that describes the mass distribution of both species. It is also typical to consider the mass sources of both species, corresponding to the local kinetics of the system, driven by inter- and intra-specific responses. As a rule, such models are special cases of evolutionary systems considered by H. Amann \cite{Amann}.

Amann's work provides a general mathematical framework for studying the qualitative features of PKS systems, such as blow-up, the extinction of one or both species, or an emergency of a particular regime of their coexistence, sometimes revealing a rather  complex spatiotemporal pattern; see \cite{Ts94, Hrstmnn, HllnPntr, BellBellTao}. A significant volume of research explores formation of such structures due to chains of bifurcations. In any case, the scenarios they describe typically begin with the instability of the simplest regime, such as the equilibrium of species with constant density (homogeneous equilibrium). In this context, the instability is the appearance or disappearance of exponentially growing modes of small perturbations caused by changing some parameters of the system. We will return to the discussion of this issue in Section~\ref{SscLnrStbDcsn}. 

{Applying the bifurcation theory to systems with taxis seems to be pioneered by work  \cite{BrzKrv} that addressed   the traveling waves driven by an autotaxis. Next, \cite{GMT,AGMTS,SapTA} studied several scenarios of the onset of complex spatiotemporal patterns starting with the Poincare-Andronov-Hopf (PAH) bifurcation in the so-called slow taxis model; \textit{cf.} Remark~\ref{RmOnFrmln}.  Chakraborty et al. \cite{Chkr} continued studying the bifurcations in this kind of model with various kinetics, and, recently, it has been revisited \cite{TaoW}. The study \cite{LeeHlnLws} has addressed the direct PKS prey-taxis--that is, the one driven by the prey density. It has  discovered the crucial dependence of the stability of the homogeneous equilibria on the system kinetics and the capability of  prey-taxis for exerting the stabilizing  effect. The instabilities  followed by the emergency of complex spatio-temporal patterns nevertheless occur for certain kinetics; \textit{cf.} \cite{HllnPntr1,BnPtrRtDpnd}.  It is worth mentioning that modeling more complex trophic chains also often uses PKS law for formalizing the interactions similar to prey-taxis. Several authors repeatedly emphasized their role in causing the instabilities and the heterogeneous pattern formation; e.g., \cite{Chaplain,Grchv}.}

{Tello \& Wrozek \cite{TllWrzk} introduced the indirect prey-taxis model and studied the onset of the secondary steady regimes. Further studies \cite{WngWng,Li} continued exploring this kind of model and discovered a rich variety of the dynamical phenomena. Several articles treat allied models, which take into account the Alee effect \cite{TtnSnTtvBnrj}, allow for more than one predator \cite{WngYngZhng}, or address the indirect pursuing-evading~\cite{WuWnGn,Xie}. In article \cite{TtnZgr}, it has been noted that, under certain conditions, predator--prey systems with indirect and slow taxis are equivalent. In particular, this equivalence holds for the systems considered in \cite{GMT,AGMTS}.}

Most of the literature addresses the PKS systems possessing translational invariance with respect to time and space, although natural objects are always heterogeneous to some extent due to natural or artificial causes. Of course, the influence of heterogeneity may be insignificant for the phenomena under consideration, but this is not always the case. For example, anthropogenic acoustic noise has a disorienting effect on marine mammals~\cite{Southall}, thermal fronts and mesoscale ocean eddies affect the distribution of fish populations~\cite{Hsu, Royer, Reese}, and the emergence of an alternative food source can affect the dynamics of omnivorous animals~\cite{Kratina}.

When an inhomogeneity excites a process on a scale that differs sharply from the other characteristic scales of the system, an effective way to take it into account is to introduce a model with small parameters and then study it using asymptotic methods, as is done in the recent works \textit{cf.} \cite{TllWrzLst,Chdh1,Chdh2,Chdh3,PgGd}, which employ two different time scales. However, it is no less logical to consider distributed small-scale inhomogeneities defined as a known field \cite{AM1,AM3,AM4}. In such a case, the issue naturally arises of the so-called induced stability or instability, or, in other words, of a change in the stability of the system as a result of an external influence under equal other conditions: For example, a comparison of the forced and free dynamics, all other things being equal, showing that one of them exhibits a stable regime, while the other does not. A classic example of induced stability is the inverted pendulum. The nonlinear theory of this phenomenon given by P. Kapitza is in fact based on the high-frequency asymptotics of nonlinear oscillations \cite{LndLfs}. {The approach of this theory has got a wide generalization, in particular, to the constrained dynamical systems \cite{Yudovich4,Yudovich3}, and has repeatedly confirmed its effectiveness upon studying the thermogravitational convection \cite{ZenSimIzvAN66,ZenMZhG68}, the dynamics of interfaces and inclusions in fluid \cite{LbmvBook,MrVld}, and the dynamics of elastic rods \cite{Tvstk,TvstkT}.} However, its implementation for predator--prey systems with taxis and more general Amann--type systems with short-wave/high-frequency inhomogeneities has not received due attention, although such an analysis could be useful for assessing the influence of inhomogeneities on the change in the stability boundaries of the system and the transition to the formation of spatiotemporal patterns, and suitable asymptotics seem to be quite accessible to homogenization methods and multiple scales expansions; see, for example, \cite{Allr-1}. To some extent, articles \cite{AM1,AM3,AM4} contribute to filling this gap. Namely, articles \cite{AM3,AM4} study the induced stability in a system with direct predator--prey taxis, where the cues are the prey density and a short-wave external signal supplied independently of the system state. Article \cite{AM1} addresses the same issue for a predator--prey system with the so-called ``slow taxis", see Remark~\ref{RmOnFrmln}. In the present article, we consider a multidimensional problem with indirect taxis, general kinetics, and an arbitrary short-wave external signal. We suppose the cue to be not the prey density but the field induced by them, and the short-wave external signal to be influencing the cue induction as well. We construct a formal asymptotic expansion for the short-wave solutions modulo solving a hierarchy of the so-called slow systems; see the end of this section for additional details.   Based on the leading term in the formal short-wave asymptotics, we examine the induced stability for this system and shed some light on the effects of the external signal on mass transfer and  pattern formation. 
{\begin{rmrk}\label{RmOnFrmln}
\emph{In particular, the aforementioned equivalence  of the slow and indirect taxis models holds true for the system considered in \cite{AM1}. In this sense, the present paper generalizes the results of article~\cite{AM1}, which addresses a model under rather strong restrictions; namely, only one spatial variable, a conservative predator, logistic prey, Lotka--Volterra kinetics of the species interactions, and the external signal in the form of a traveling wave} 
\end{rmrk}} %\newpage
\subsection{Formal set-up}
%%%%%%%%%%%%%%%%%%%%%%%%%%%%%%%%%%%%%%%%%%%%%%%%%%%%%%%%%%%%%%%%%%%%%%%%%%%%%%%%%%%%%%%%%%%%%%%%%%%%%55
%%%%%%%%%%%%%%%%%%%%%%%%%%%%%%%%%%%%%%%%%%%%%%%%%%%%%%%%%%%%%%%%%%%%%%%%%%%%%%%%%%%%%%%%%%%%%%%%%%%%%%%%%%%%%%%%%%%%%%%
The general predator--prey system with indirect taxis and an external signal has the form:
\begin{eqnarray}
&p_t+\Div({\chi} p \nabla s-\delta_1 \nabla p)=pf(p,q),&
\label{EqPrdTrns}\\
&q_t-\Div(\mu\nabla q)= qg(p,q), &
\label{EqPryDnst} \\
&s_t-\Div(\delta_2 \nabla s)={{{\kappa_1}}} h+{{\kappa}} q -\nu s.&
\label{EqSgnl}
\end{eqnarray}
Here the distribution densities of predators and prey, $p=p(x,t)\ge 0$, $q=q(x,t)\ge 0$,  are unknown functions of the spatiotemporal coordinates, $x\in\mathbb{R}^n$, $t>0$. Given functions, $f$ and $g$, determine the densities of the mass sources—that is, the local kinetics of the system. We assume them to be analytic in the cone $\{p>0,q>0\}$ and continuous up to its boundary, except for the origin. Within this cone, Eq \eqref{EqPrdTrns} determines the predator balance through local kinetics, diffusion, and PKS taxis induced by a cue field of intensity $s=s(x,t)$, where the coefficient ${\chi}=\const >0$ measures the predator's sensing to the cue.

The additional equation \eqref{EqSgnl} regards the cue field intensity. It shows that there is an external field,  the intensity of which is determined by the known function $h=h(x,t)$, and this external signal influences the cue production by the prey. Hence, the external signal contributes to the PKS flow not directly, but through exerting an effect on the cue field induction. Additionally, the decay of the cue under the influence of the environment is taken into account. The coefficients $\nu=\const>0$, ${\kappa=\const}>0$ denote the rates of the cue decay and reproduction, respectively,  and coefficient ${\kappa_1=\const}>0$ measures the sensing of the cue reproduction  to  the external signal. The diffusion coefficients of predators, cue, and prey are denoted as $\delta_1=\const>0$, $\delta_2=\const>0$, and $\mu=\const>0$.

In the following discussion, we consider the system \eqref{EqPrdTrns}--\eqref{EqSgnl} as dimensionless and assume that the following equalities are satisfied:
\beear
&
\chi=1,\ \kappa_1=1,\ \mu=1,
&
\label{NrmPrm}\\
&
g(0,q)\to 1,\ q\to+0,\quad g(0,1)=0,\quad g_p(p,1)\to -1,\ p\to+0.
&
\label{NrmRctn}
\eear
These assumptions are satisfied with appropriate scaling and therefore do not cause a loss in generality. In more detail: The transformation of the dimensional form of  system \eqref{EqPrdTrns}--\eqref{EqSgnl} into a dimensionless form consists of the substitution
\beq
\breve{x}=Lx,\ \breve{t}=Tt,\ \breve{p}=pP,\breve{q}=qQ,\ s=\breve{s}S,\ h=\breve{h}H,
\eeq
where $\breve{z}$ denotes the dimensional version of variable $z$, and the positive quantities $L,T,P,Q,S,H$ are the characteristic scales of the variables $\breve{x},\breve{t},\breve{p},\breve{q},\breve{s},\breve{h}$ as yet undefined. Hence, the general dimensionless system \eqref{EqPrdTrns}--\eqref{EqSgnl} has the form
\beq
P T^{-1}p_t+PL^{-1}\Div(\breve{\chi} p SL^{-1}\nabla s-\breve{\delta}_1 L^{-1}\nabla p)=Pp\breve{f}(Pp,Qq),&
\\ 
& 
QT^{-1}q_t-QL^{-1}\Div(\breve{\mu} L^{-1}\nabla q)= Qq\breve{g}(Pp,Qq), 
&
\\ 
& 
T^{-1}S s_t-L^{-1}S\Div(\breve{\delta}_2 L^{-1}\nabla s)=\breve{\kappa_1} Hh+\breve\kappa Q q -\breve{\nu}S s.
\eeq
In applications, function $g$ often reads as $g_0(q)-pg_1(p,q)/q$, where function $g_0$ is positive in a right neighborhood of zero including zero itself and changes sign exactly once; at this, $g_1(p,q)>0$ for $p>0,q>0$, $\lim\limits_{p\to+0}g_1(p,q)>0\ \forall\,q>0$, and $g_{1p}(p,q)=O(1)$, $p\to+0,\,\forall\, q>0$; \textit{cf.} \cite{TtnTrFn} or, for more detail, \cite{ArdG}.
 Therefore, suppose that $\exists\ q_1>q_0>0$, such that 
\beq
\breve{g}(0,q)>0\ \forall q\in [0,q_0),\quad \breve{g}(0,q_0)=0,\ \breve{g}(0,q)<0\ \forall\, q\in (q_0,q_1),\quad \lim\limits_{p\to+0} \check{g}_p(p,1)<0.
\eeq
Then, by choosing $Q$ first, then $T$, and then $P$, we satisfy the following conditions:
\beq
\breve{g}(0,Q)=0,\quad T\lim\limits_{q\to+0}\breve{g}(0,Qq)=1,\quad PT\lim\limits_{p\to+0}(\pr\breve{g}/\pr\breve{p})(Pp,Q)\to-1.
\eeq
In particular, for Lotka--Volterra kinetics with the logistic reproduction of prey, the dimensionless expression for the function $g$ has the form $g=1-q-p$. Next, we choose $L$ first, then $S$, and then $H$ in a way that the following equalities hold:
\beq
{\mu}=L^{-2}T\breve{\mu}=1,\quad \chi=L^{-2}S\breve{\chi}=1,\ \kappa_1=\breve{\kappa}_1 TH/S=1.
\eeq

In accordance with the above-stated goal of the study, we assume that
\begin{eqnarray}&\label{ShrtWvSgnl}
h=h(x,t,x/\delta,t/\delta),\quad \delta_1=\mu_1\delta,\quad \delta_2=\mu_2\delta,
&\end{eqnarray}
where the number $\delta>0$ is small, and the function $h=h(x,t,\xi,\tau)$ is $\mathrm{C}^{\infty}$-smooth on $\mathbb{R}^{n+1}\times\mathbb{R}^{n+1}$ and periodic in the coordinates $\xi_1,\ldots,\xi_n,\tau$ with periods $\ell_1,\ldots\ell_n,\ell_0$. We seek for a formal expansion
\begin{eqnarray}&\label{AsmpExpnsn}
(p,q,s)=\sum\limits_{k=0}^{K-1}({p}_k,q_k,s_k)(x,t,\xi,\eta)\delta^k+O(\delta^K),\ \delta\to+0,\ K\in \mathbb{N},
&\end{eqnarray}
where the order of approximation, $K$, can be as high as one wishes, and the coefficients, ${p}_k,q_k,s_k$, are periodic in fast coordinates, $\xi_1,\ldots,\xi_n,\tau$, with periods $\ell_1,\ldots\ell_n,\ell_0$. This periodicity entails that  each term is  one order smaller than the preceding one.

Based on the two-scale expansions method, we formulate an iterative procedure that allows for constructing arbitrarily  many terms of expansion \eqref{AsmpExpnsn}. As usual, each step of two-scale expansion requires solving a couple of PDE systems; one of them is in fast variables, $\xi,\tau$, only, while the other one is in the slow variables, $x,t$, only. Moreover, the solution to the former system must satisfy the periodicity conditions stated above. We show that the fast system arising at each step has a solution possessing the required periodicity, provided the slow system of the previous step has a solution, and resolving  the slow system  completes evaluating the corresponding coefficient in the asymptotic  expansion.  However, we refrain from proving the existence of solutions to the whole chain of slow systems. Although we are quite optimistic regarding it, it evidently should be the subject of a separate paper. 

At the same time, we state explicitly the leading term and the leading slow system that ensures the existence of the next term. We reveal a class of very simple explicit solutions to the leading slow system called quasi-equilibria, for which we address the issue of the induced stability and instability relying on this leading approximation only. In this regard, we follow the widely accepted procedure dating back to Kapitza's inverted pendulum theory, the outcomes of which get additional reliability thanks to the possibility for constructing the higher-order approximations.

The article is organized as follows. In Section~\ref{ScSttng}, we introduce some notation and auxiliary tools. In Section~\ref{ScAsmpt}, we expose the leading approximation of the asymptotics, derive the so-called leading slow system governing the leading approximation on average, and address several adjacent issues. The higher asymptotic approximations are addressed  in Appendix B. The rest of Section~\ref{ScAsmpt} covers the following: (i) discussing the terms in the leading slow system through which the external signal affects the slow dynamics and their explicit expressions for some particular classes of the external signals; (ii) a class of very simple asymptotic solutions called quasi-equilibria; (iii) the linearization of the leading slow system around a general quasi-equilibrium and the particular classes of the external signals, which allow for expressing explicitly the terms through which the external signal exerts its effect on the structure of the linearized leading slow system. In Section~\ref{ScIndStb}, we address the stability and instability induced by the external signal. We start with introducing some specific terminology and then perform the linear stability analysis using the so-called normal modes. Furhermore, we  formulate some criteria for the stability and instability in terms of the control parameters of the model, including the ones depending on the signal only. Based on this, we reveal effects of stabilization and destabilization due to the external signal. Section~\ref{ScCncl} exposes several conclusive remarks. For better dynamics of the presentation, we supplied several assertions with the proofs' sketches only and moved the detailed proofs to Appendix A. Also, we set Appendix C that contains some details of the analysis presented in section~\ref{ScIndStb}.
\section{Auxiliaries}
\label{ScSttng}
\noindent
Let $D_{x,t}\byd \{(x,t)\in\mathbb{R}^{n+1}:\ x\in \Omega,\,t\in (t_0,t_1)\}$, where $t_0\le 0<t_1$ and $\Omega$ is a domain in   $\mathbb{R}^{n}$ or an $n$-dimensional torus. In what follows, all considered  functions of the variables $(x,t,\xi,\tau)$ are assumed to be periodic and square-summable in the fast coordinates   $\xi_1,\ldots,\xi_n,\tau$ with periods $\ell_1,\ldots,\ell_n, \ell_0$ for arbitrarily fixed slow variables $(x,t)\in D_{x,t}$. Exceptions to this rule are stipulated separately.
We define and denote the operations and operators of full and partial averaging as follows:
\begin{eqnarray}
&
(Mf)(x,t)\byd\langle f \rangle(x,t)\byd\lim\limits_{L\to\infty}\frac{1}{(2L)^{n+1}}\int\limits_{(-L,L)^{n+1}}f(x,t,\xi,\tau)\,d\xi d\tau,
&\label{DfAvrfng}
\\
& (M^\xi f)(x,t,\tau)\byd\langle f \rangle^\xi (x,t,\tau)\byd\lim\limits_{L\to\infty}\frac{1}{(2L)^{n}}\int\limits_{(-L,L)^{n}}f(x,t,\xi,\tau)\,d\xi,&
\label{DfSptAvrgng} \\
& (M^\tau f)(x,t,\xi)\byd\langle f \rangle^\tau (x,t,\xi)\byd\lim\limits_{L\to\infty}\frac{1}{(2L)}\int\limits_{(-L,L)}f(x,t,\xi,\tau)\,d\tau.&
\label{DfTmprlAvrgng}
\end{eqnarray}
Let $\mathbb{T}_{\xi,\tau}^{n+1}$ be the $n+1$-torus obtained by factorising $\mathbb{R}^{n+1}$ over the lattice 
$
\mathbb{Z}_\ell^{n+1}\byd \ell_1\mathbb{Z}\times \ldots\times{\ell_n}\mathbb{Z}\times \ell_{0}\mathbb{Z}.
$
Taking into account the periodicity assumptions given above, 
we identify the fast coordinates $\xi,\tau$ with coordinates on the torus $\mathbb{T}_{\xi,\tau}^{n+1}$, and henceforth by a function of the variables $(x,t,\xi,\tau)$, we mean a function on $D_{x,t}\times \mathbb{T}_{\xi,\tau}^{n+1}$. Accordingly, the averaging operations  introduced above allow for  considering as operations over the  functions on the torus and define operators $M$, $M^{\xi}$, $M^{\tau}$,  bounded in ${\mathrm{L}}_2(\mathbb{T}_{\xi,\tau}^{n+1})$, or more precisely, orthogonal projectors. 
In particular, 
\beq
\mathrm{range}\left(\mathrm{I}-M\right)=\{u\in{\mathrm{L}}_2(\mathbb{T}_{\xi,\tau}^{n+1}):\, \langle u \rangle=0 \}\byd \widetilde{\mathrm{L}}_2(\mathbb{T}_{\xi,\tau}^{n+1}).
\eeq
Moreover,
$ M^\xi M^\tau= M^\tau M^\xi =M$ or, in other notation, 
$\langle\langle f \rangle^\xi \rangle^\tau=\langle\langle f \rangle^\tau \rangle^\xi=\langle f \rangle$.

For convenience of the subsequent presentation, we conformally change the standard metric in ${\mathrm{L}}_2(\mathbb{T})$, where $\mathbb{T}$ is a torus of arbitrary dimension; namely, we set
$
\|f\|_2=\langle ff^*\rangle,
$
where the superscript $^*$ denotes complex conjugation and the angle brackets denote averaging over the torus.

Furthermore, the use of the notation for the fast variables as subscripts in the notation for a dependent variable indicates the partial derivative of this dependent variable with respect to the corresponding fast variables. Similarly, the use of the notation for fast variables as subscripts to denote a differential operation means that this operation must be performed by differentiating with respect to the fast variables while the slow variables are fixed.  Differential operations with respect to the fast variables will also be denoted by a wide tilde above, for example, $\wtlnbl$.

Let $\eps>0$. We set 
\begin{eqnarray}
& \mathcal{H}_{\eps}\byd \pr_\tau-\eps\widetilde{\Delta},\quad\mathcal{D}\byd ({{\wtlnbl}\,\cdot\,} \nabla +{\nabla\,\cdot\,}{\wtlnbl});&
\label{Ntn3}\\
& 
\mathrm{H}^{s,2}\byd \{u\in{\mathrm{L}}_2(\mathbb{T}_{\xi,\tau}^{n+1}):\, \sum\limits_{} (k^{2s}+m^4)|\hat{u}_{km}|^2<\infty\quad (k,m)\in \mathbb{Z}_\ell^{n+1}\};
&
\label{DfDomH}\\
&
\hat{u}_{k,m}=\langle u\mathrm{e}^*_{k,m}\rangle ,\quad \mathrm{e}_{k,m}={\exp(i(k\tau+m\xi))},\quad (k,m)\in \mathbb{Z}_\ell^{n+1}.
&
\nonumber
\end{eqnarray}
Next, on the domain $\mathrm{H}^{1,2}\subset {\mathrm{L}}_2(\mathbb{T}_{\xi,\tau}^{n+1})$, we define the operator $\mathcal{H}=\mathcal{H}_\eps:{\mathrm{L}}_2(\mathbb{T}_{\xi,\tau}^{n+1})\to {\mathrm{L}}_2(\mathbb{T}_{\xi,\tau}^{n+1})$ by expression \eqref{Ntn3}. According to this definition, the operator $\mathcal{H}_\eps$  is densely defined  and closed.

In what follows, by a spectral projector of an operator, we mean  the projector onto the spectral subspace of this operator that commutes with it. The operator $\mathcal{H}_{\eps}$ is  the heat operator on the torus; it is well-known that $\ker \mathcal{H}_{\eps}$ is its spectral subspace consisting of constant functions.  In other words, no associated eigenfunctions correspond to  the zero eigenvalue.  The corresponding spectral projector is the averaging operator $M$. The right inverse  
operator $\mathcal{H}_{\eps}^{-1}$ is bounded from $\widetilde{\mathrm{L}}_2(\mathbb{T}_{\xi,\tau}^{n+1})\to \widetilde{\mathrm{L}}_2(\mathbb{T}_{\xi,\tau}^{n+1})$. Thus, 
 $\forall\, v\in {\mathrm{L}}_2(\mathbb{T}_{\xi,\tau}^{n+1}):\ \langle v\rangle=0$, the general solution to the equation $\mathcal{H}_{\eps}u=v$ on the torus $\mathbb{T}_{\xi,\tau}^{n+1}$ has the form:
$
u=\langle u\rangle +\mathcal{H}_\eps^{-1}v, \quad \langle\mathcal{H}_{\eps}^{-1}v\rangle=0,
$
where $\langle u\rangle\in \ker \mathcal{H}_\eps$ is a free parameter. 

Quite analogously to $\mathcal{H}_{\eps}$, the operators $\pr_\tau:{\mathrm{L}}_2(\mathbb{T}_{\xi,\tau}^{n+1})\to {\mathrm{L}}_2(\mathbb{T}_{\xi,\tau}^{n+1})$ and $\widetilde{\Delta}:\,{\mathrm{L}}_2(\mathbb{T}_{\xi,\tau}^{n+1})\to {\mathrm{L}}_2(\mathbb{T}_{\xi,\tau}^{n+1})$ are defined, and their right inverse operators, $\widetilde{\Delta}^{-1} v_1$ and  ${\widetilde{\pr_\tau}}^{-1} v_1$, are bounded in the spaces $(\mathrm{I}-M^\xi){\mathrm{L}}_2(\mathbb{T}_{\xi,\tau}^{n+1})$ and $(\mathrm{I}-M^\tau){\mathrm{L}}_2(\mathbb{T}_{\xi,\tau}^{n+1})$, respectively. Thus, $\forall\,v_1\in {\mathrm{L}}_2(\mathbb{T}_{\xi,\tau}^{n+1}):\ M^{\xi}v_1=0$ and $\forall\,v_2\in {\mathrm{L}}_2(\mathbb{T}_{\xi,\tau}^{n+1}):\ M^{\tau}v_2=0$, the equations
$
\widetilde{\Delta} u_1=v_1, \quad \pr_\tau u_2=v_2
$
have solutions, and any of them can be represented in the form 
$
u_1=M^\xi u_1+\widetilde{\Delta}^{-1} v_1,
$
$
u_2=M^\tau u_2+{{\pr_\tau}}^{-1} v_2,
$
where the functions $M^\xi u_1\in \ker \widetilde{\Delta}^\xi$ and $M^\tau u_2\in \ker\pr_\tau$ remain undetermined.
\begin{dfntn}\label{DfL(w)}
\emph{\textsf{For  $\eps>0$ and  a vector function $\vw\in \mathrm{C}^\infty\left(D_{x,t}\times \mathbb{T}_{\xi,\tau}^{n+1},\mathbb{R}^n\right)$,  define the linear operator $\mathcal{L}=\mathcal{L}(\vw,\eps):\, \mathrm{L}_2(\mathbb{T}_{\xi,\tau}^{n+1})\to\mathrm{L}_2(\mathbb{T}_{\xi,\tau}^{n+1})$ acting on the domain  $\mathrm{H}^{1,2}$  by mapping $ u\mapsto \mathcal{H}_{\eps}u+ {\wtlnbl}\,\cdot\,(u\vw).$}} 
\end{dfntn}
\begin{lemma}\label{LmmExst}
\emph{\textsf{For all $(x,t)\in D_{x,t}$, the spectral projector ${P}$ onto $\ker \mathcal{L}$ has the form}}
$$
{P}:u\mapsto \langle u\rangle \astphi,\quad \astphi\in \ker \mathcal{L},\quad \langle \astphi \rangle=1.
$$
\end{lemma}
\noindent\textbf{Proof.} 
Let $\mathcal{L}^*$ denote the operator adjoint to $\mathcal{L}$ and let $\mathcal{J}:{\mathrm{L}}_2(\mathbb{T}_{\xi,\tau}^{n+1})\to{\mathrm{L}}_2(\mathbb{T}_{\xi,\tau}^{n+1})$ be the action of the inversion $(\xi,\tau)\mapsto(-\xi,-\tau)$. We set 
$
\breve{\mathcal{L}}^*\byd\mathcal{{J}}\mathcal{L}^*\mathcal{{J}}.
$
Then  $\ker \breve{\mathcal{L}}^*=\ker\mathcal{L}^*$ and  
$$
\breve{\mathcal{L}}^*:\varphi\mapsto (\pr_\tau+\vw{\wtlnbl}) \varphi-\eps \widetilde{\Delta}\varphi,\quad\forall\varphi\in  \mathrm{H}^{1,2}.
$$
Since the resolvents of the operators $\breve{\mathcal{L}}^*$, ${\mathcal{L}}^*$ and ${\mathcal{L}}$ are compact, their spectra are discrete and the solvability of the equations ${\mathcal{L}}^*u=f$ and ${\mathcal{L}}u=f$ obeys the Fredholm alternative. Therefore, $\dim\ker{\mathcal{L}}= \dim\ker{\mathcal{L}}^*$ and $\dim \mathrm{range}\,{P}= \dim \mathrm{range}\,{P}^*$, that is, the dimensions of the spectral subspaces also coincide.

The equation $(\pr_\tau+\vw{\wtlnbl}) \varphi-\eps \widetilde{\Delta}\varphi=0$ is subject to the strong maximum and minimum principles (see, e.g., ~\cite{Nrnbrg,Lnds}). Therefore any solution to it on the torus is constant. Thus,
$$
\ker \breve{\mathcal{L}}^*=\ker\mathcal{L}^*=\{\varphi=\const\}.
$$
Suppose there is an associated eigenfunction to the zero eigenvalue of operator ${\mathcal{L}}^*$. Then there exists a solution to the equation on the torus
$
{\mathcal{L}}^*\psi=1,
$
but this contradicts the one-sided minimum principle. Hence, $\dim\,\mathrm{range}\,{P}=1$. Moreover, the assumption that $\langle u \rangle=0$ for some $u\in \mathrm{Ker}\,\mathcal{L}\setminus\{0\}$ would imply the existence of a solution to the equation $\mathcal{L}^*\psi=1\neq 0$, which would contradict what we have proved above. Furthermore, directly by definition, 
$${P}^2u=\langle{P}u\rangle\astphi =\langle u\rangle\langle \astphi\rangle \astphi={P}u$$ 
and
$
\mathcal{L}{P}={P}\mathcal{L}=0.
$
$\square$
\begin{rmrk}\label{RmOnPhi}
\emph{A useful by-product of Lemma~\ref{LmmExst} is the implicit mapping
\begin{eqnarray}&\label{DfPhi}
{\Phi}(\cdot,\eps):\vw\mapsto \astphi,\quad \astphi\in \ker\mathcal{L}( {\vw},\eps),\quad \langle\astphi\rangle=1.
&\end{eqnarray}
In view of Lemma~\ref{LmmExst}, the right inverse operator $\mathcal{L}^{-1}( {\vw},\eps):\widetilde{\mathrm{L}}_2(\mathbb{T}_{\xi,\tau}^{n+1})\to \widetilde{\mathrm{L}}_2(\mathbb{T}_{\xi,\tau}^{n+1})$ is defined, and $\forall\, f\in \widetilde{\mathrm{L}}_2(\mathbb{T}_{\xi,\tau}^{n+1})$, the general solution to the equation $\mathcal{L}( {\vw},\eps)u=f$ has the form $u=\bar{u}\astphi+\mathcal{L}^{-1}( {\vw},\eps)f$. Moreover $\langle u \rangle=\bar{u}$.}
\end{rmrk}
\section{Asymptotics}  
\label{ScAsmpt}
\noindent
Here we return to the system consisting of Eqs \eqref{EqPrdTrns}--\eqref{EqSgnl} with a general short-wave external signal~\eqref{ShrtWvSgnl}. We discuss the leading term of the asymptotics \eqref{AsmpExpnsn} (it is of zero order in the small parameter), address the terms of orders 1 and 2,  and clarify the conditions under which they do exist. As a result, we derive the so-called leading slow system, which governs the leading term on average.  For addressing the  higher-order terms, see  Appendix~II. To focus our considerations more precisely, we additionally assume that $\langle h\rangle=0$.

Let $f$ be some function in $D_{x,t}\times\mathbb{T}^{n+1}_{\xi,\tau}$, for example, one of the coefficients in expansion \eqref{AsmpExpnsn}. We adopt the brief notation $\overline{f}=\langle f \rangle=Mf$, $\widetilde{f}=(\mathrm{I}-M)f$; in particular, 
$\tilde{h}=h$, $\bar{h}=0$.

%thus,
%\begin{eqnarray*}
%&
%f=\bar{f}+\tilde{f}=f^\bullet+f^\circ,\quad f=\bar{f}+\widetilde{f}^\flat +\widetilde{f}^\xi,\quad %\tilde{f}=\widetilde{f}^\xi+\widetilde{f}^\flat,
%&
%\end{eqnarray*} 

After introducing fast variables $\xi,\tau$ and taking into account convention \eqref{NrmPrm}, system \eqref{EqPrdTrns}--\eqref{EqSgnl} takes the following form:
\begin{eqnarray}
&\delta \left(\mathcal{H}_{ \mu_1}p+ {\nabla}\,\cdot\,(p{\wtlnbl}{s})+{ {\wtlnbl}\,\cdot\,}(p{\nabla}{s})\right)+{ {\wtlnbl}\,\cdot\,}(p{\wtlnbl}{s})=
&
\nonumber\\
&
\delta^2\left(pf(p,q)-(\pr_t-\mu_1\mathcal{D})p-{ {\nabla}\,\cdot\,}(p\nabla {s})+\delta \mu_1\Delta p\right),
&
\label{EqPrdFstVrbl0}\\
&\left({ -\widetilde{\Delta}}+\delta \left(\pr_\tau- \mathcal{D}\right)\right) q= \delta^2\left(qg(p,q)- (\pr_t-  \Delta){q}\right),&
\label{EqPryDnstFstVrbl0}\\
&(\mathcal{H}_{ \mu_2}+\delta(\pr_t-\mu_2 \mathcal{D})){s}=\delta \left(\tilde{h}+{{\kappa}} q-\nu s+\delta \mu_2 \Delta s\right).&
\label{EqDrvFstVrbl0}
 \end{eqnarray}
Substitution of expansion \eqref{AsmpExpnsn} into Eqs \eqref{EqPrdFstVrbl0}--\eqref{EqDrvFstVrbl0}
and grouping terms of order $k=0,1,2,\ldots$ leads to a chain of equations for the coefficients of this expansion, which must be solved on $D_{x,t}\times\mathbb{T}^{n+1}_{\xi,\tau}$. 
To start with, we write down the systems of equations corresponding to $k=0,1,2$.  
\begin{eqnarray}
&\text{For}\ k=0:&
\nonumber\\
&{ {\wtlnbl}\cdot} (p_0{\wtlnbl}{q}_0) =0,\quad\widetilde{\Delta}q_0=0,\quad \mathcal{H}_{ \mu_2}s_{0}=0.&
\label{EqPrdCfcOrd0-0}\\
 &\text{For}\ k=1:&
\nonumber\\ 
&
\mathcal{L}(\vw_0, \mu_1)p_0+{\wtlnbl}\,\cdot\,( p_1{\wtlnbl}s_0)=0, \quad\ \text{where}\ \vw_{0}= \nabla s_{0}+{\wtlnbl}{s}_{1},
&
\label{EqPrdCfcOrd1-0}\\
&{ \widetilde{\Delta}}q_1=\left(\pr_\tau- \mathcal{D}\right)q_{0},&
\label{EqPryCfcOrd1-0}\\
&\mathcal{H}_{ \mu_2}s_{1}=\tilde{h} +{{\kappa}} q_0-\nu s_0 - (\pr_t-\mu_2\mathcal{D})s_0.&
\label{EqDrvrCfcOrd1-0} 
 \\
 &\text{For}\ k=2:&
 \nonumber\\
 & 
\mathcal{L}(\vw_0, \mu_1)p_1+{\wtlnbl}\,\cdot\,( p_2{\wtlnbl}s_0)=\Psi,\ \text{where}
&
\label{EqPrdCfcntK=2}\\
&
\Psi=p_0f_0-{(\pr_t-\mu_1\mathcal{D})}p_{0}-{\nabla\,\cdot\,}(p_0\vw_0)-{ {\wtlnbl}\,\cdot\,}(p_0\vw_1),\
&
\nonumber\\
&
\vw_1={\wtlnbl}{s}_{2}+{\nabla}{s}_{1},\  f_0=f(p_0,q_0),
&
\nonumber\\
& {- \widetilde{\Delta}}{q}_{2}={q}_0g_{0}-\left(\pr_\tau-\mathcal{D}\right)q_{1}-(\pr_t-\Delta){q}_0,\ g_0=g(p_0,f_0),&
\label{EqPryCfcntK=2}\\
&
\mathcal{H}_{ \mu_2} s_{2}={{\kappa}} q_{1}-\nu s_{1}-(\pr_t-\mu_2\mathcal{D})s_{1}+\mu_2\Delta{s}_0.
&
\label{EqDrvrCfcntK=2}
 \end{eqnarray}
\subsection{Leading approximation}\label{SScLdngApprx}
\noindent
Solutions to Eq \eqref{EqPrdCfcOrd0-0} have the form:
\begin{eqnarray}&\label{S0Q0Frst}
s_0=\bar {s}(x,t),\quad q_0=\bqx_0(x,t,\tau)
&\end{eqnarray}
while the function $p_0$ remains completely unknown.

In view of equalities \eqref{S0Q0Frst}, Eq \eqref{EqPrdCfcOrd1-0} takes the form $\mathcal{L}(\vw_0, \mu_1)p_0=0$. By Lemma~\ref{LmmExst},
\begin{eqnarray}
&p_0=\bar{p}\Phi(\vw_0, \mu_1),\quad \bar{p}=\bar{p}(x,t),&
\label{P00}
\end{eqnarray}
where the map $\Phi$ was introduced by Remark~\ref{RmOnPhi}. The function $\bar{p}$ plays the role of a constant of integration.
The compatibility condition of Eq \eqref{EqPryCfcOrd1-0} has the form $\pr_\tau\bqx_0=0$, and therefore $q_0=\bar{q}$, that is, function $\bar{q}=\bar{q}(x,t)$ is another ``constant of integration". The compatibility condition for Eq~\eqref{EqDrvrCfcOrd1-0} reads
\begin{eqnarray}
&\bar {s}_{t} +\nu \bar{s}={{\kappa}} {\bar{q}}.&
\label{EqDrvrLdngSlw}
\end{eqnarray}

If $q_0=\bar{q}$ and $s_0=\bar{s}$, where $\bar{s}$ is a solution to Eq \eqref{EqDrvrLdngSlw}, then system \eqref{EqPryCfcOrd1-0}--\eqref{EqDrvrCfcOrd1-0} has the solution
\begin{eqnarray}
&
q_{1}=\bqx_1(x,t,\tau),\quad s_{1}=\tilde{s}_1+ \bar{s}_1,\quad \tilde{s}_1=\mathcal{H}_{ \mu_2}^{-1}\tilde{h},
&
\label{Q1&S1}
\end{eqnarray}
where functions $\bar{s}_1$ and $\bqx_1$ remain unknown. 

Thus, in the leading approximation, the predators respond to the short-wave external signal by a short-wave parcel \eqref{P00}. Its slow amplitude,  $\bar{p}=\bar{p}(x,t)$, represents the mean density of predators—that is, $\langle p_0\rangle=\bar{p}$. This parcel  undergoes a twofold slow modulation by  altering both the amplitude, $\bar{p}$, and the transport velocity field, $\vw_0=\nabla\bar s+\wtlnbl\tilde{s}_1$. While its fast component, $\wtlnbl\tilde{s}_1$, is completely determined by the external signal, Eq \eqref{EqDrvrLdngSlw} links  the slow one, $\nabla\bar s$,   to the leading approximation of the prey density, $\bar{q}$, which is not yet determined.

The system governing the slow modulation of the leading approximation follows from the compatibility conditions of  system
\eqref{EqPrdCfcntK=2}--\eqref{EqDrvrCfcntK=2}. Taking into account the identities ${\wtlnbl}{s}_{0}=0$ and ${\wtlnbl}{q}_{1}=0$, see \eqref{S0Q0Frst} and \eqref{Q1&S1}, we write these conditions as
\begin{eqnarray}
&
\bar{p}_t+{ \nabla\cdot}\left(\bar{p} \left(\nabla\bar {s} + \bar{\mathrm{V}}(\nabla\bar {s})\right)\right)=\bar{p}\bar{f}(\bar{p},\bar{q},\nabla\bar {s}),\quad \text{where}\quad \bar{f}(\bar{p},\bar{q},\nabla \bar {s})\byd \langle {\astphi} f({\astphi}\bar{p},\bar{q})\rangle,
\quad
&
\label{EqPrdLdngSlw}
\\
&
\bar{\mathrm{V}}(\nabla \bar{s})\byd \langle {\astphi}{\wtlnbl}\tilde{s}_1\rangle,\quad {\astphi}\byd{\Phi}(\nabla\bar{s}+\wtlnbl\tilde{s}_1, \mu_1),
\quad \tilde{s}_1=\mathcal{H}_{ \mu_2}^{-1}\tilde{h},
&\label{DfF&V}
\\
&\bqx_{1\tau}=M^{\xi}\left(\bar{q}{g}_0 -(\pr_t- \Delta)\bar{q}\right),&
\label{EqPryLdngSlw-0}
\\
&
\bar{s}_{1t} +\nu \bar {s}_1={{\kappa}} {\bar{q}}_{1}+ \mu_2\Delta \bar{s},\ \text{where}\ {\bar{q}}_{1}=\langle q_1\rangle,\ {\bar{s}}_{1}=\langle s_1\rangle,
&
\label{EqDrvrSlwOrdr1}
\end{eqnarray}
where $g_0$ is defined in \eqref{EqPryCfcntK=2}. Equation \eqref{EqPryLdngSlw-0} is compatible under the condition
\begin{eqnarray}
&(\pr_t- \Delta)\bar{q}=\bar{q}\bar{g}(\bar{p},\bar{q},\nabla\bar {s}),\quad \text{where}\quad \bar{g}(\bar{p},\bar{q},\nabla\bar {s})=\langle g({\astphi}\bar{p},\bar{q})\rangle.&
\label{EqPryLdngSlw}
\end{eqnarray}
\begin{dfntn}\label{DfLdngSlwSst}
\emph{\textsf{By the leading slow system, we mean the system of Eqs \eqref{EqDrvrLdngSlw}, \eqref{EqPrdLdngSlw}, \eqref{DfF&V}, and \eqref{EqPryLdngSlw} with respect to the mean values of the leading approximation, $\bar{p}=\bar{p}(x,t)$, $\bar{q}=\bar{q}(x,t)$, and $\bar{s}=\bar{s}(x,t)$.}}
\end{dfntn}
\begin{rmrk}\label{RmOnMmr}
\emph{The external signal influences the leading slow system, changing both the kinetics and the tactical transport of predators—that is, transforming $f\to \bar{f}, g\to \bar{g}$ and adding the drift term, ${ \nabla\cdot}(\bar {p}\bar{\mathrm{V}}(\nabla\bar s))$, to the predator mass balance. Hence, both the averaged transport and the kinetics are nonlocal. A notable exception is $p$-linear kinetics, that is, where the functions $pf(p,q)$, $qg(p,q)$ are linear (not necessarily homogeneous) in $p$. Then they are preserved under homogenization, so that $\bar f=f(p,q)$ and $\bar g=g(p,q)$. In particular, Lotka--Volterra kinetics is $p$-linear, as are a number of other classical models of interspecies reactions, \emph{cf.}, \cite{TtnTrFn}.}
\end{rmrk} 
The following theorem summarizes the above analysis.
\begin{theorem}\label{ThrmLdngSlwSst}
\emph{\textsf{Let the mean values of the leading approximation, $\bar{p}=\bar{p}(x,t), \bar{q}=\bar{q}(x,t)$, $\bar{s}=\bar{s}(x,t)$, give a solution to the leading slow system. Then the systems of Eqs \eqref{EqPrdCfcOrd0-0}--\eqref{EqDrvrCfcOrd1-0} completely determine the coefficients $p_0,q_0,s_0$, as well as the coefficients $q_1,s_1$ up to their mean values, $\langle q_1\rangle$ and $\langle s_1\rangle$. More precisely, $p_0=\bar{p}{\astphi}$, $s_0=\bar{s}$, $q_0=\bar{q}$, $q_1=\bqx_1=\tilde{q}_1+\bar{q}_1$, $\tilde{q}_1=\pr_\tau^{-1} (M^\xi-M){g_0}$, $s_{1}=\tilde{s}_1+\bar{s}_1$, $\tilde{s}_1=\mathcal{H}_{ \mu_2}^{-1}\tilde{h}$.
Moreover, Eqs \eqref{EqPrdCfcntK=2} and \eqref{EqPryCfcntK=2} become compatible, and their solutions, $p_1$ and $q_2$, have the form
$
q_2=\bqx_2 - \bar{q}\widetilde{\Delta}^{-1}(I-M^\xi){g_0},
$
$
p_1=\bar{p}_1{\astphi}+\mathcal{L}^{-1}(\vw_0,\mu_1)\widetilde{\Psi}. 
$
Also, 
$
\tilde{s}_2=\mathcal{H}^{-1}_{\mu_2}\left(\kappa \tilde{q}_{1}-\nu \tilde{s}_{1}-(\pr_t-\mu_2\mathcal{D})\tilde{s}_{1}\right),
$
provided that Eq \eqref{EqDrvrSlwOrdr1} holds true.  The projections $\bqx_2=M^\xi q_2$, $\bar{s}_2$, $\bar{q}_1$, $\bar{s}_1$, and $\bar{p}_1$ remain undetermined in any case.}}
\end{theorem}

\subsubsection{Drift and stationary signal}
\label{SsscDrft}
\noindent
Within this subsubsection, the notation of $\tilde{s}$ replaces the notation of $\tilde{s}_1=\mathcal{H}_{ \mu_2}^{-1}\tilde{h}$, where $\tilde{h}$ is the external signal (see formulae \eqref{Ntn3} and \eqref{Q1&S1}).  
 Here, the signal  may depend on slow variables too. We do not indicate this explicitly to simplify the notation.

We set
\beear
\efr\byd\exp(\tilde{s}/ \mu_1),\quad \efr_\ast\byd \efr\langle \efr\rangle^{-1}.
\label{DfExpFn} 
\eear
 Finding $\astphi={\Phi}(\bar{\vu}+\wtlnbl\tilde{s}, \mu_1)$ consists of solving the stationary equation
\beear
\wtlnbl\cdot(\astphi(\widetilde{\nabla} \tilde{s}+\bar{\vu})- \mu_1\wtlnbl\astphi)=0.
\label{EqPhiStd}
\eear
The substitution
$
\astphi=\efr\psi\langle\efr\psi\rangle^{-1}
$
reduces it to the equation:
\beear
\wtlnbl\cdot(\efr(\psi\bar{\vu}- \mu_1\widetilde{\nabla}\psi))=0.
\label{EqPsi}
\eear
For $\bar{\vu}=0$, this equation is solved explicitly, and we conclude that $\astphi=\efr_\ast$. Then  
$
\bar{\mathrm{V}} (0)= \langle \efr_*\widetilde{\nabla}\tilde{s}\rangle= \mu_1\langle\widetilde{\nabla}\efr_\ast\rangle=0.
$
Thus, $\bar{\mathrm{V}}(0)=0$ under the condition that the external signal is stationary and the cue intensity is everywhere constant. (In the general case, this is not so, even if the cue intensity is constant, see Section~\ref{SsscTrWv}).

There exists a class of external signals allowing one to explicitly express the field $\mathrm{\bar{V}}$; namely,
\beear
&
\tilde{h}=\tilde{h}_1(\xi_1)+\ldots+\tilde{h}_{n}(\xi_n)\Rightarrow \tilde{s}=\tilde{s}_1(\xi_1)+\ldots+\tilde{s}_{n}(\xi_n),\quad \langle\tilde{h}_j\rangle=\langle\tilde{s}_j\rangle=0,\ j=1...n.
&
\label{EqH=SumHiStd}
\eear
Let the external signal have the form of \eqref{EqH=SumHiStd}. Then the solution to Eq \eqref{EqPsi} reads
\beear
\psi=\psi_1(\xi_1)\ldots\psi_n(\xi_n), \quad \psi_i(y)=(\bar{u}_i- \mu_1\pr_{y})^{-1}\efr^{-1}_i(y),\quad \efr_i(y)=\exp(\tilde{s}_i(y)/ \mu_1),\ y\in \mathbb{R}.
\label{EqPsi-I}
\eear
Further, $\langle\efr\psi\rangle\mathrm{\bar{V}}(\bar{\vu})=\langle\efr\psi\wtlnbl\tilde{s}\rangle= \mu_1\langle\psi\wtlnbl\efr\rangle$ by the definition of this field and the function $\tilde{s}$. Hence
\beear
&
\mathrm{\bar{V}}(\bar{\vu})=-\frac{ \mu_1\langle \efr \wtlnbl\psi\rangle}{\langle\efr\psi\rangle},\ \text{that is,}\quad \mathrm{\bar{V}}_i(\bar{\vu})=\mathrm{\bar{V}}_i(u_i)=-\frac{ \mu_1\langle \efr_i\pr_i\psi_i\rangle}{\langle \psi_i{\efr}_i\rangle},\ \quad\text{where}\ \pr_i\byd \pr_{\xi_i},\ i=1...n.
&
\eear
We simplify this as follows
$ -\mu_1\langle \efr_j\pr_j \psi_j\rangle= 
\langle \efr_j(\efr_j^{-1}-\bar{u}_j\psi_j)\rangle$. Hence,  
\beear
\mathrm{\bar{V}}_j(\bar{\vu})=\frac{1}{\langle \efr_j \psi_j\rangle}-\bar{u}_j,\quad j=1...n.
\label{EqDrftVlct=}
\eear
In view of formulae \eqref{EqDrftVlct=}, Eq \eqref{EqPrdLdngSlw} takes the following form:
\beear 
\bar{p}_t+\nabla\cdot(\bar{p}\bar{\vv})=\bar{p}\bar{f}(\bar{p},\bar{q},\nabla\bar{s}),\
\bar{v}_j=\frac{1}{\langle \efr_j (\bar{s}_{{x}_j}- \mu_1\pr_{\xi_j})^{-1}\efr^{-1}\rangle}.\label{EqPrdLdngStdy}
\eear
\begin{rmrk}\label{RmOnPole}
\emph{ For each, $j=1...n$ the right-hand side of formula \eqref{EqDrftVlct=} admits an analytic continuation from the real axis of the plane $\bar{u}_j$ to some strip covering it, and moreover $\bar{V}_j(\bar{u}_j)=O(\bar{u}_j)$, $\bar{u}_j\to+0$. This follows from the fact that zero is a simple pole of the resolvent of the operator $\pr:\mathrm{L}_2(\mathbb{S})\to \mathrm{L}_2(\mathbb{S})$, where the notation of $\mathbb{S}$ is for  the circumference.  The corresponding spectral projector is $f\mapsto \langle f\rangle 1$, where the angle brackets denote averaging over $\mathbb{S}$. Hence,
$
\lambda\langle \efr_j (\lambda- \mu_1\pr_{\xi_j})^{-1}\efr_j^{-1}\rangle\to \langle \efr_j \rangle\langle\efr_j^{-1}\rangle,
$ 
$
\lambda\to 0,
$
where $\langle \efr_j \rangle\langle\efr_j^{-1}\rangle>1$ by virtue of the convexity of the exponent function. From this and from formulae~\eqref{EqPrdLdngStdy}, there follow the inequalities $\bar{v}_j/\bar{s}_{x_j} < 1$ for small $\bar{s}_{x_j} $, $j=1...n$.  This indicates the possibility for suppression of tactical transport since, in the original system, the speed of tactical transfer is \emph{equal} to $\nabla \bar{s}$ {(see \eqref{EqPrdTrns} and \eqref{NrmPrm})}.}
\end{rmrk}
\begin{rmrk}\label{RmOnExpSprss}
\emph{We set $h_i(\xi)=A_{0i}h_{0i}(\xi)$ for some $i=1...n$, where  $\langle h^2_{0i}\rangle =1$. Then it is logical to regard the quantity $A_i\byd A_{0i} / \mu_2$ as the effective amplitude of the $i$-th term in the external signal given by equality \eqref{EqH=SumHiStd}. The Laplace asymptotics of the integrals expressing the mean values $\langle\efr_i\rangle$ and $\langle\efr^{-1}_i\rangle$ show exponential growth of each of them as $A_i\to\infty$, so that the above-mentioned suppression of tactic transport is strengthened exponentially as the amplitude of the external signal grows.}
\end{rmrk}
\subsection{Quasi-equilibria}\label{SscQsEq}
 \noindent
By the equilibria of system \eqref{EqPrdTrns}--\eqref{EqSgnl},   we mean particular solutions such that both the species densities and the cue intensity are identically constant — that is,
\begin{eqnarray}&\label{EqlbrDnst}
{s}= {s}_e\equiv\const,\quad {q}= {q}_e\equiv \const,\quad {p}= {p}_e \equiv \const.
&\end{eqnarray}
For the existence of equilibria, the external signal must be constant. Then the equilibrium species densities $p_e,q_e$ and the cue intensity $s_e$ satisfy the equations:
\begin{eqnarray}
{p}_e{f}({p}_e,{q}_e)=0,\quad {q}_e {g}({p}_e,{q}_e)=0,\quad h+ {\kappa} q_e-\nu s_e=0, \quad h=\const.
\label{PeQe}
\end{eqnarray}
If a system possesses a complete translational invariance, then the possibility for its equilibrium is expected, and this is not the case otherwise. However, if a small-scale inhomogeneity does not undergo modulation on the basic and greater scales, it is natural to consider the so-called quasi-equilibria---that is, the short waves that do not undergo slow modulation. 

In order to exclude a slow modulation of the external signal, we adopt the condition
\beear
&
\tilde{h}=\tilde{h}(\xi, \tau),\ \nabla\tilde{ h}=\pr_t\tilde{h}\equiv 0.
&
\label{NoSlowMdln}
\eear 
 In the leading approximation, every quasi-equilibrium corresponds to a suitable equilibrium (understood in the usual sense) of the leading slow system. Indeed, the existence of a quasi-equilibrium entails the equalities 
\begin{eqnarray}&\label{QEqlbr}
\bar{s}= \bar{s}_e\equiv\const,\quad \bar{q}= \bar{q}_e\equiv \const,\quad \bar{p}= \bar{p}_e \equiv \const, \ \phi_*=\phi_*(\xi,\tau)
&\end{eqnarray}
(this follows from the expressions of the leading approximation by formulae \eqref{S0Q0Frst}--\eqref{P00}).
Thus, these densities have to satisfy the following equations:
\begin{eqnarray}
&
\bar{p}_e\bar{f}_e(\bar{p}_e,\bar{q}_e)=0,\quad\bar{q}_e\bar{g}_e(\bar{p}_e,\bar{q}_e)=0,\quad \nu \bar{s}_e -{{\kappa}} {\bar{q}_e}=0,
& 
\label{BrPeQe}\\
&
\bar{f}_e(\bar{p},\bar{q})\byd\left.\bar{f}(\bar{p},\bar{q},\nabla\bar{s})\right|_{\nabla\bar{s}=0},
\quad 
\bar{g}_e(\bar{p},\bar{q})\byd\left.\bar{g}(\bar{p},\bar{q},\nabla\bar{s})\right|_{\nabla\bar{s}=0}.
&
\label{BrFeGe}
\end{eqnarray} 
The possibility of equilibrium of the leading slow system does not contradict the possibility of drift in view of condition \eqref{NoSlowMdln}, owing to which    the drift field in Eq $\eqref{EqPrdLdngSlw}$  is constant. 
\begin{rmrk}\label{RmOnQEqVsEq}
  \emph{In the general case, the equations of quasi-equilibria---that is, Eq \eqref{BrPeQe}---differ from the equations of equilibria---that is, Eq \eqref{PeQe}---since the functions $\bar{f}_e$ and $\bar{g}_e$ do not coincide with functions $f$ and $g$ given originally. However, they coincide in the case of $p$-linearity (see also Remark~\ref{RmOnMmr}). Thus, for $p$-linear kinetics, every quasi-equilibrium corresponds to an equilibrium of the original system with no signal and vice versa---that is, ${\bar{s}_e}= {s}_e$, ${\bar{q}_e}= {q}_e$, and $ \bar{p}_e={p}_e$.}
\end{rmrk}
\subsection{Linearization} 
\noindent
Define a pair of linear functions $\bar{a}^e_i:\mathbb{R}^2\to \mathbb{R}$, $i=1,2$, and a pair of vector fields $\bar{\vb}^e_i$, $i=1,2$, by setting
\begin{eqnarray}
& \bar{a}^e_i({\zeta}_1,{\zeta}_2)=\langle a_{i}(\astphi{\zeta}_1,{\zeta}_2)\rangle, \ \ i=1,2,%\quad \astphi=\Phi_0(\nabla\bar{s}). \widetilde{a}_i({\zeta}_1,{\zeta}_2)=a_i({\zeta}_1,{\zeta}_2)-\bar{a}_i({\zeta}_1,{\zeta}_2), &
&
\label{DfBr-a-LnFn}\\
&
{\vb}^e_i\nabla\bar{\psi}= \langle a_i(\mathcal{L}^{-1}(\wtlnbl\tilde{s}_1,\mu_1){\wtlnbl}\cdot(\astphi\nabla\bar{\psi}),0)\rangle,\ i=1,2,
&
\label{DfBrB-Fld}
\end{eqnarray}
for an arbitrary smooth test function $\psi=\psi(x,t)$. Here, functions $a_i$ have the form   
\begin{eqnarray}
%&a_{1,m}\byd \sum\limits_{j=1..m-1}p_jf_{m-j}=Op_{m-1},\quad %a_{2,m}\byd\sum\limits_{j=1..m-1}s_jg_{m-j} = Op_{m-1},
&
a_{1}({\zeta}_1,{\zeta}_2)=d(pf(p,q))({\zeta}_1,{\zeta}_2)\left|_{p=p_0,q={q}_0}\right.,\quad
a_{2}({\zeta}_1,{\zeta}_2)=d(qg(p,q))({\zeta}_1,{\zeta}_2)\left|_{p=p_0,q={q}_0}\right.,
&
\label{Df-a-LnFn}\\
&
 \text{where}\quad p_0=\bar{p}_e\astphi,q_0=\bar{q}_e.
&
\nonumber
\end{eqnarray} 
\begin{lemma}\label{LmmLnrzn} 
\emph{\textsf{The linearization of the leading slow system near its equilibrium $(\bar{p}_e,\bar{q}_e,\bar{s}_e)$ has the form
\begin{eqnarray}
& 
{p}_t+{ \nabla\cdot}\left( {p}{\vc}^e + \bar{p}_e \mathcal{T}\nabla {s}\right)+\bar{p}_e\bar{\vb}^e_1\nabla {s}=\bar{a}^e_1( {p}, {q}),
&
\label{EqLnrPrdLdngSlw}
\\
&
(\pr_t- \Delta) {q}+\bar{p}_e\bar{\vb}^e_2\nabla {s}=\bar{a}^e_2( {p}, {q}),
&
\label{EqLnrPryLdngSlw}
\\
& 
{s}_{t} +\nu {s}= {\kappa}{ {q}},\ \text{where}
&
\label{EqLnrDrvrLdngSlw}
\\
&
\vc^e=\bar{\mathrm{V}}(0),\quad \mathcal{T}_{i,j}=\delta_{ij}-\left\langle\astphi_i\tilde{s}_{1\xi_j}\right\rangle,\, i,j=1...n,\quad \tilde{s}_1= \mathcal{H}^{-1}_{ \mu_2}\tilde{h},\quad \bar{p}_e\vb^e_i=\bar{\vb}_i,
&
\label{EqDfT&Ce}\\
&
\astphi_i={\mathcal{L}}^{-1}(\wtlnbl\tilde{s}_1,\mu_1)\astphi_{\xi_i},\ \langle \astphi_i\rangle=0,\quad i=1...n,
&
\label{DfPhiAstI}
\end{eqnarray}
where  $\bar{a}^e_i$ are the linear functions \eqref{DfBr-a-LnFn} and ${\vb}^e_i$ are the vector fields \eqref{DfBrB-Fld}.}}
\end{lemma}
\noindent{\it Proof.} The key point is the differentiation of the mapping $\check{\Phi}: y\mapsto \Phi(\vy+\wtlnbl\tilde{s}_1,\mu_1)$, where $y\in \mathbb{R}^n$, the map $\Phi$ is defined by formula \eqref{DfPhi}, and the differentiation is performed at $y=0$, while we identify the element $y$ and the constant vector field $\vy$ on the torus $\mathrm{T}_{\xi,\tau}^{n+1}$. At the formal level it is fairly obvious that $d\check{\Phi}(v)=-v_i\astphi_{i}$. A more rigorous and detailed study of this differentiation can be found in Appendix A. Let us also dwell on the computation of the sensing tensor, $\mathcal{T}$. Let $\bar{\vg}:(x,t)\mapsto\bar{\vg}(x,t)\in \mathbb{R}^n$.  In view of formula \eqref{DfF&V}, we have 
$
\bar{\mathrm{V}}(\bar{\vg})\nabla\bar{\psi}= \langle (\check{\Phi}\circ \bar{\vg})\wtlnbl\tilde{s}_1\nabla\bar{\psi}\rangle, \quad
$
where  $\bar{\psi}$  is a test function of the slow variables.  Hence,
$
\nabla\bar{\psi}\bar{\mathrm{V}}^\prime\bar{\vv}= \langle d\check{\Phi}(\bar{\vv})\wtlnbl\tilde{s}_1\nabla\bar{\psi}\rangle,
$
where $\bar{\mathrm{V}}^\prime$ denotes the differential of the map $\bar{\mathrm{V}}$. Thus, differentiation of the map $\bar{\mathrm{V}}$ leads to a tensor field with components $-\left\langle\astphi_i\tilde{s}_{1\xi_i}\right\rangle$. The terms involving the linear functions $a_i^e$ and the vector functions $\vb_i^e$ are obtained analogously, upon linearization of the maps 
$
F_1:(\bar{p},\bar{q},\bar{s})\mapsto \bar{p}\bar{f}(\bar{p},\bar{q},\nabla\bar{ s}),\quad F_2: (\bar{p},\bar{q},\bar{s})\mapsto \bar{q}\bar{g}(\bar{p},\bar{q},\nabla\bar{ s}).
$  $\square$
\begin{rmrk}\label{RmOnCnstntCfcnts}
\emph{Since the external signal, by assumption, does not undergo  slow modulation--that is, equalities \eqref{NoSlowMdln} hold--the linearized system \eqref{EqLnrPrdLdngSlw}--\eqref{EqLnrDrvrLdngSlw}  possesses translational invariance. Indeed, it turns out that  $\mathcal{T}=\const$, $\vc^e=\const$, $\vb_i^e=\const$, $i=1,2$.}
\end{rmrk}
\begin{rmrk}\label{RmOnP-Lnr-1}
\emph{For $p$-linear kinetics, $\vb_i^e=0$. Indeed, let $v$ be an arbitrary function on the torus of fast variables. In view of $p$-linearity, $\langle a_i(v\zeta_1,0)\rangle=\langle v\rangle a_i(\zeta_1,0)$, $i=1,2$. Hence, the mean values are zero in formula \eqref{DfBrB-Fld}.  A similar argument shows that $\bar{a}^e_i=a^e_i$, $i=1,2$, where  $a^e_i$ are   the linear functions~\eqref{Df-a-LnFn} at $p_0=\bar{p}_e$, $q_0=\bar{q}_e$;  \emph{cf.},  Remark~\ref{RmOnQEqVsEq}.}
\end{rmrk}
\subsubsection{Stationary signal}
\label{SsscStSgnl&Lnrzn}

From here on, the notation of $\tilde{s}$ replaces  $\tilde{s}_1=\mathcal{H}_{ \mu_2}^{-1}\tilde{h}$ unless cancelling this convention is proclaimed.  This should be kept in mind upon operating functions $\efr$ and $\efr_\ast$, \emph{cf.,}  equalities \eqref{DfExpFn}.  

In this subsection, we supplement assumption \eqref{NoSlowMdln} by the following one:
\beq
h=\tilde{h}(\xi) \Rightarrow\tilde{s}=\tilde{s}(\xi).
\eeq 
Let such a signal be fixed. Then  the components $\mathcal{T}_{ij}$ are computed by averaging over the torus $\mathbb{T}_{\xi}^n$ and the operator $\mathcal{L}=\mathcal{L}(\wtlnbl \tilde{s},\mu_1)$ is replaced by its stationary version $w \mapsto \wtlnbl\cdot(w\wtlnbl \tilde{s}-\mu_1\wtlnbl w)$, acting in $\mathrm{L}_2(\mathbb{T}_{\xi}^{n})\to \mathrm{L}_2(\mathbb{T}_{\xi}^{n})$, with the previous notation preserved. Lemma~\ref{LmmExst} remains valid in this case  and under its conditions, $\astphi=\efr_*$ (see Section~\ref{SsscDrft} concerning the last equality). Accordingly, the right inverse $\mathcal{L}^{-1}( {\vw},\eps):\widetilde{\mathrm{L}}_2\left(\mathbb{T}_{\xi}^{n}\right)\to \widetilde{\mathrm{L}}_2\left(\mathbb{T}_{\xi}^{n}\right)$ is defined,  where the tilde above  indicates the subspace of functions with zero mean in ${\mathrm{L}}_2$. 

By assumption,  $\vc^e=\bar{\mathrm{V}}(0)=0$ (see Section~\ref{SsscDrft}). However, the sensing tensor cannot be computed explicitly with no additional assumptions.  Therefore, we restrict ourselves to the class of external  signals \eqref{EqH=SumHiStd}. Then $\efr=\efr_1(\xi_1)\efr_2(\xi_2)...\efr_n(\xi_n)$, $\efr_i=\mathrm{e}^{\frac{\tilde{s}_i}{\mu_1}}$, and the substitution 
$
\astphi_i=\psi_i\efr_\ast,\quad\psi_i=\psi_{i}(\xi_i),
$
reduces equation $\mathcal{L}(\wtlnbl \tilde{s},\mu_1)\astphi_i=\astphi_{\xi_i}$ to  the equations 
$
 \mu_1\psi_{i\xi_i}=\frac{1}{\efr_i\langle\efr^{-1}_i\rangle}-1,\ i=1...n.
$
Each of them  has a unique solution   such that $\langle \astphi_i\rangle=0$, where $\astphi_i=\psi_i\efr_*$.   Further,
\beear
&
\mathcal{T}_{j,k}=\delta_{j,k}-\langle\astphi_j\tilde{s}_{\xi_k}\rangle=\delta_{j,k}+ \mu_1\langle\psi_{j\xi_k} \efr_*\rangle=\delta_{j,k}
+\delta_{j,k}\left\langle\left(\frac{1}{\efr_j\langle\efr^{-1}_j\rangle}-1\right) \efr_*\right\rangle\Rightarrow
&
\\
&
\mathcal{T}_{j,k}=\frac{\delta_{j,k}}{\langle\efr_j\rangle\langle\efr^{-1}_j\rangle},\quad {j,k=1\ldots n}.
&
\label{EqTrTnsrSttnr}
\eear
\begin{rmrk}\label{RmOnExtnsn1}
 \emph{One can deduce the result in \cite{AM1} pertaining to the effect of a stationary signal from expressions \eqref{EqTrTnsrSttnr} in the particular case of  $n=1$.}
\end{rmrk}
\begin{rmrk}\label{RmOnLnrExpSprssng}
\emph{Another way to derive expressions \eqref{EqTrTnsrSttnr} consists in linearizing expression \eqref{EqDrftVlct=} for the drift velocity. One way or another, these expressions agree with Remark~\ref{RmOnExpSprss} in the sense that they reveal the suppression of tactic transport in the direction of the axis $Ox_j$ with exponential sharpness as the amplitude $A_j$ of the component $h_j$ of the external signal \eqref{EqH=SumHiStd} grows.} 
\end{rmrk}
Consider  the case of a general stationary signal free from restriction \eqref{EqH=SumHiStd}. 
\begin{lemma}\label{LmmTrnspTnsr}
 \emph{\textsf{ Let $h=h(\xi)$. Then the sensing tensor $\mathcal{T}$ is symmetric and its eigenvalues belong to the {semi-interval} $(0,1]$.}}
\end{lemma}
\noindent{\it Proof.} Let  $\bar{\vu}=(u_1,..,u_n)$ be a vector field on the torus  associated with {a constant vector} in $\mathbb{R}^n$. The keynote in the proof is that the quadratic form $u_iu_j\mathcal{T}_{ij}$ allows for evaluating  using the  quadratic form of some projector, which acts in Hilbert's space of vector fields on the torus with a specially selected metric; see Appendix A for details.
\begin{rmrk} \label{RmOnMinEgnTrTn}
\emph{The maximum eigenvalue of the sensing tensor can reach  unity provided that   the external signal, $\tilde{h}$, is invariant with respect to translations along some subspace in  $\mathbb{R}_\xi^n$, which in this case is the eigen-subspace of the unity.  However,  a generic stationary signal generates a sensing tensor such that each of its eigenvalues is strictly less than unity—that is, the tactic transport is suppressed to some extent; \emph{cf}. Remark~\ref{RmOnMinEgnTrTn1}. In this sense, the effect of a generic stationary signal is analogous to that produced by signal \eqref{EqH=SumHiStd}, although we cannot yet assert exponential decay of the eigenvalues of the sensing tensor matrix as generic external signal intensity grows.}
\end{rmrk}
\subsubsection{Traveling waves}
\label{SsscTrWv}

In this subsection, we supplement assumption \eqref{NoSlowMdln} with the following one:
\beq
h=\tilde{h}(\eta),\ \eta=(\eta_1,\ldots,\eta_n),\quad \eta_i=\xi_i-\omega_i\tau, \quad |\omega_i|\ell_0 =\ell_i, \ i=1...n,
\eeq
or, possibly, $\omega_i=0$. The function $\tilde{h}$ is assumed $\ell_i$-periodic in each of the self-similar variables, while the restrictions on the wave speeds, $\omega_i$, entail time periodicity with a given period, $\ell_0$. With these self-similar variables, we associate all functions and differential operations considered in this section. In particular, the operator $\mathcal{L}(\wtlnbl\tilde{s},\mu_1)$ is replaced by its self-similar version:
\beq
\mathcal{L}(\wtlnbl\tilde{s},\mu_1): w\mapsto\nabla\cdot(w(\wtlnbl\tilde{s}-\bar{\vomega})- \mu_1\wtlnbl\astphi),\quad \bar{\vomega}=(\omega_1,\ldots,\omega_n).
\eeq
Then the equation for the function $\astphi$ coincides with Eq \eqref{EqPhiStd} up to the replacement of $\bar{\vu}$ by $-\bar{\vomega}$.
To use the explicit solution to the latter presented in Section~\ref{SsscDrft}, assume the  external signal is a sum of simple traveling waves:
\beear
\tilde{h}=\tilde{h}(\eta)=\tilde{h}_1(\eta_1)+\ldots+\tilde{h}_n(\eta_n),\quad \langle\tilde{h}\rangle=\langle\tilde{h}_j\rangle=0,\ j=1...n,
\label{H=TrWv}
\eear
where  the means  are computed over the torus $\mathbb{T}_\eta^{n}$, as in all subsequent formulae of this section. Then $\tilde{s}=\tilde{s}_1(\eta_1)+\ldots+\tilde{s}_n(\eta_n)$, where $\tilde{s}_j=-(\omega_j+\mu_2\pr_j)^{-1}\pr^{-1}_j\tilde{h}_j$, $\langle\tilde{s}_j\rangle=0$, and $\pr_j=\pr_{\eta_j}$, $j=1...n$. 
\begin{rmrk}\label{RmOnTrWv}
\emph{Useful information about the behavior of the functions $\tilde{s}_j$ as  $\omega_j\to 0$ and $\omega_j\to\infty$ is given by their a priori estimates  $\omega^2\langle (\pr^\ell_j \tilde{s}_j)^2\rangle\le \langle(\pr^{\ell-1}\tilde{h}_j)^2\rangle$, $\ell=1,2,...$, as well as $\langle (\pr^\ell_j (\tilde{s}_j-\tilde{S}_j))^2\rangle\le \omega^2\langle(\pr^{\ell}\tilde{S}_j)^2\rangle$, where $S_j=-\pr_j^{-2}\tilde{h}_j/\mu_2$.}  
\end{rmrk}

Consider the residual drift, $\vc^e=\langle\astphi\nabla \tilde{s}\rangle$ (see \eqref{DfF&V} and \eqref{EqDfT&Ce}) for a fixed external signal in the form~\eqref{H=TrWv}. Then
\beear\label{EqRsdDrftTrWv}
\vc^e=(c_1,\ldots,c_n),\quad c_j=\omega_j-\frac{1}{\langle \efr_j \psi_j\rangle},
\eear
where the functions $\psi_j$ and $\efr_j$ are expressed by formulae~\eqref{EqPsi-I}, up to the replacement of $\bar{u}_i$ by $-\omega_i$, $\psi_j$ by $-\psi_j$, and $\xi_j$ by $\eta_j$. Thus,   $\vc^e\neq 0$, generally speaking. Hence, a nonzero  drift with uniform cue distribution  is the generic case in the class of external signals \eqref{H=TrWv}, in contrast to the class of stationary external signals. 
\begin{rmrk}\label{RmOnC_jOm_jToInfty}
\emph{ It turns out that $c_j=O(\langle (\pr_j s_j)^2\rangle/\omega_j)$, $\omega_j\to\infty,\ \ j=1...n$ for a fixed signal $h$ in the form~\eqref{H=TrWv}.  These estimates follow from the equalities
\beear\label{EqEjPsjAvrg}
\mu_1\langle\efr_j \psi_j\rangle= \sgn(\omega_j)\int\limits_{\eta>0}\mathrm{e}^{-\frac{|\omega_j|\eta}{\mu_1}}
\langle{\efr}\sigma_{j,\eta}\efr^{-1}\rangle d\eta,\ \text{where}\ (\sigma_{j,\eta} \efr)(t)= \efr(t-\eta\sgn(\omega_j)),\ j=1...n,
\eear
which follow  directly from the definition of the functions $\psi_j$.  In addition, 
\beq
c_j=\omega_j\left(1-\left(\langle\efr_j\rangle\langle\efr_j^{-1}\rangle\right)^{-1}+O(\omega_j)\right),\ \omega_j\to 0,\ \ j=1...n,
\eeq
where $\langle\efr_j\rangle\langle\efr_j^{-1}\rangle>1$, \emph{cf.} Remarks~\ref{RmOnPole} and \ref{RmOnExpSprss}. }
\end{rmrk}

Let us compute  the sensing tensor.  We start with the solution to the equations $\mathcal{L}\astphi_j=\phi_{\ast\eta_j}$, $\langle\phi_{\ast\eta_j}\rangle=0$, $j=1...n$, where $\mathcal{L}=\mathcal{L}(\wtlnbl\tilde{s},\mu_1)$. The solutions have the form: $\astphi_j={\efr}\check{\psi}_j(C_j\psi_j-\theta_j)\langle\efr \psi\rangle^{-1}$, where $\theta_j=(\omega_j+ \mu_1\pr_j)^{-1}\psi_j$,  $\check{\psi}_j={\psi_1\ldots\psi_n}{\psi_j}^{-1}$, and we set $C_j={\langle{\efr}_j\theta_j\rangle}{\langle{\efr}_j\psi_j\rangle}^{-1}$ to obey the condition $\langle\astphi_j\rangle=0$, where the denominator is nonzero by virtue of equality \eqref{EqEjPsjAvrg}.  
Further, $\mathcal{T}_{jk}=\delta_{jk}\, -\,\langle\astphi_j\pr_k\tilde{s}_k\rangle$. Proceeding with the computations, we arrive at the equalities 
\beear
&
\mathcal{T}_{ij}=\delta_{ij}\langle\efr_j\theta_j\rangle\langle\efr_j\psi_j\rangle^{-2},\quad i,j=1...n.
&
\label{EqTrTnsTrWv}
\eear
\begin{rmrk}\label{RmOnTij}
  \emph{By definition of the functions $\theta_j$, $j=1...n$, 
  \beear\label{EqEjThtjAvrg}
  {\mu_1}^{2}\langle\efr_j\theta_j\rangle=\int\limits_{\eta>0}\eta \mathrm{e}^{-\frac{|\omega_j|\eta}{\mu_1}}
\langle{\efr}\sigma_{j,\eta}\efr^{-1}\rangle d\eta,\ \text{where}\ (\sigma_{j,\eta} \efr)(t)= \efr(t-\eta\sgn(\omega_j)),\ j=1...n,
  \eear 
{and these equalities entail the inequalities }$\mathcal{T}_{ii}>0$, $i=1...n$, as well as the estimates  
\beear
&
\mathcal{T}_{ij}=\delta_{ij}\left(\left(\langle\efr_j\rangle\langle\efr_j^{-1}\rangle\right)^{-1}+O(\omega_j)\right),\ \omega_j\to+0,\ i,j=1...n,
&
\label{EqTijOmgTo0}\\
&
\mathcal{T}_{ij}=\delta_{ij}\left(1+{\langle (\pr_j s_j)^2\rangle}\omega_j^{-2}(1+O(\omega_j^{-1}))\right),\ {\omega_j}\to \infty,\ i,j=1...n.
&
\label{EqTijOmgToInfty}
\eear
}
\end{rmrk}
\begin{rmrk}\label{RmOnInrvtng}
\emph{Equalities \eqref{EqTijOmgTo0} indicate the possibility of \emph{suppression} of tactic transport in the direction of propagation of each of the waves constituting the signal \eqref{H=TrWv}. Indeed, without the signal, $\mathcal{T}_{ij}=\delta_{ij}$, but $\mathcal{T}_{jj}<1$ for a nonzero amplitude of the wave $h_j$ and a sufficiently small speed of its propagation, $\omega_i$. Moreover, the leading term in estimate  \eqref{EqTijOmgTo0} can be arbitrarily small for a sufficiently large amplitude $A_j$ of the wave $h_j$, and this effect is exponential; \emph{cf.}, Remark~\ref{RmOnExpSprss}. At the same time, {equalities} \eqref{EqTijOmgToInfty} {show} that $\mathcal{T}_{jj}>1$ for large $\omega_j$. Therefore,  an increase in the propagation speed of the wave that is a component of the external signal \eqref{EqTijOmgTo0}  \emph{strengthens} the tactic transport in the direction of propagation of this wave.} 
\end{rmrk}
\section{Induced stability and instability }
\label{ScIndStb}
\noindent
In this section, we compare the stability of equilibria and quasi-equilibria, using the leading slow
system to study the latter ones. We assume the coexistence of species---that is, there exists both an equilibrium
and a quasi-equilibrium---and the following inequalities hold true:
\beear\label{ExstCoExst}
p_e>0,\quad q_e>0,\  \bar{p}_e>0,\quad \bar{q}_e>0.
\eear
\subsection{Terminology}
\label{SscLnrStbDcsn}
\noindent
The linear stability analysis of systems possessing translational invariance (see Remark~\ref{RmOnCnstntCfcnts}) uses the so-called \emph{normal modes} of perturbations, which have the following form:
\begin{equation}\label{EgnMdsGnrl}
\hat{a}\exp(\lambda t+i \hx x),\quad \hat{a}\in \mathbb{C}^\ell,\quad \lambda \in \mathbb{C},\quad \hat{x}\in \mathbb{R}^n,
\end{equation}
where $\ell$ is the number of unknowns in the linearized system, $\hat{a}=\hat{a}(\hx)$ is the unknown amplitude, and $\lambda=\lambda(\hx)$ is the eigenvalue of the corresponding spectral stability problem, which for each particular wave vector, $\hx$, reduces to an algebraic eigenvalue problem owing to the separation of variables. A normal mode is called {\it stable (unstable, neutral)} if the real part of the corresponding eigenvalue $\lambda$ is negative (positive, equal to zero). It is also customary to say that an \emph{equilibrium is stable (unstable, neutral)} if each of its eigenmodes is stable (there exists an unstable (neutral) mode).

Let the system together with its equilibrium depend smoothly on $N$ parameters. Then the spectral stability problem depends on the same parameters, and, hence, each normal mode depends on them, as well as on the wave vector $\hx$. The subset of the $n+N$-dimensional parameter space on which neutral modes exist is called \emph{neutral}. The boundary of the projection of this set along the $\hx$-subspace onto the subspace of the remaining $N$ parameters forms a \emph{critical set} separating the stability region from the instability region in the $N$-dimensional parameter space of the family of equilibria.

The neutral set of parameters, as a rule, is a union of smooth strata of various dimensions. If a parameter point belongs to a stratum of maximal dimension $N+n-1$, then there exists either a simple pair of complex conjugate modes with $\mathrm{Re}\,\lambda=0,\,\mathrm{\IM}\,\lambda \neq 0$, or a simple real mode with $\lambda =0$. Points of the first and second types form two types of disjoint $N+n-1$-dimensional strata. Strata of the first (second) type are called oscillatory (monotone). The sections of the monotone and oscillatory strata by 2D planes spanned by various couples of the parameters are widely known as the neutral curves.

Figure~\ref{FgNtrlCrtcl} illustrates a possible shape of the neutral set of a one-parameter family ($n=N=1$). (The figure has no relation to any particular stability problem.) It shows 4 neutral strata of maximal dimension 1, and their junction point—a stratum of zero dimension. The red strata are, say, monotone, and the blue ones are oscillatory. The critical set consists of a single value of the parameter, denoted by $A_{cr}$. The instability region in the one-dimensional parameter space is given by the inequality $A>A_{cr}$.

\begin{figure}[H]
\centering
\includegraphics[scale=0.3]{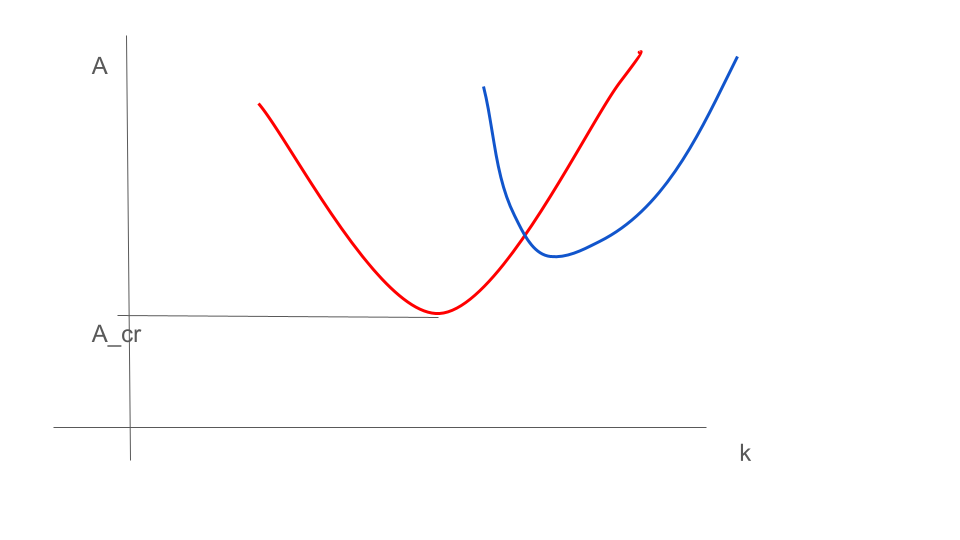}
\caption{Scheme of the neutral set in the case of $N=n=1$, where the parameter is  $A$, and the wave number is $k$.}
\label{FgNtrlCrtcl}
\end{figure}
If a smooth path transversally intersects a neutral stratum of maximal dimension, then a stable eigenmode becomes unstable, or vice versa. This phenomenon is called the 	``onset of instability," and we speak of oscillatory (monotone) instability if the intersection occurs with an oscillatory (monotone) stratum. The onset of instability is an important indicator of local bifurcations: Oscillatory instability is a precursor of the excitation of self-oscillations according to the Andronov--Hopf bifurcation scenario, and monotone instability is a precursor of branching off the secondary stationary regimes, perhaps inhomogeneous, from the equilibrium.  All these transitions are of considerable interest, in particular, in connection with the study of pattern formation in the dynamics of cells or populations, {see the articles cited in the Introduction.} 
\subsection{Normal modes}
\label{SsscNrmlMds}
\noindent
In what follows, we omit those proofs,  the reductions of which  to   elementary computations are obvious. In any case, they are quite accessible using computer algebra, \emph{e.g.}, \emph{Maple}. However, for the sake of transparency, we put  several key expressions in Appendix C. 

Consider a quasi-equilibrium satisfying condition \eqref{ExstCoExst}, the corresponding linearization of the leading slow system \eqref{EqLnrPrdLdngSlw}--\eqref{EqLnrDrvrLdngSlw}, and its normal modes \eqref{EgnMdsGnrl}. Let us  introduce artificial diffusion terms into the corresponding spectral stability problem  for addressing possible effects of small diffusion of the predator and the cue (see the original system \eqref{EqPrdTrns}--\eqref{EqSgnl}).   As a result, for each wave vector $\hx$, there arises the eigenvalue problem
\begin{eqnarray}
&
\left(
  \begin{array}{ccc}
\lambda +\mathrm{i}k {c}+ k^2\delta_1- a_{11}& -a_{12} &\mathrm{i}k p_e b_1-\check{\chi} p_e k^2\\
    -a_{21} &  \lambda - a_{22}+ k^2& \mathrm{i}k p_eb_2 \\
    0 & -\kappa & \lambda +\nu+k^2\delta_2\\
  \end{array}
\right)\left(
         \begin{array}{c}
          \hat{p}\\
           \hat{q}\\
           \hat{s}\\
         \end{array}
       \right)=0,
&
\label{EqMtrxEgnVlPrblm}
\end{eqnarray}
where  $(\hat{p},\hat{q},\hat{s})=\hat{a}$ is the unknown amplitude of the normal mode, 
$
k=|\hx|$,  $c\byd{\vc^e\cdot \boldsymbol{\hx}}/{k},\ b_i\byd {\bar{\vb}^e_i\cdot \boldsymbol{\hx}}/{k},\ i=1,2,\ \check{\chi}\byd {\boldsymbol{\hx}}\cdot \mathcal{T}\boldsymbol{\hx}/{k^2}$. {Furthermore,} 
$a_{ij}$, $i,j=1..2$, are the coefficients of the linear functions $\bar{a}^e_i$ which, together with the vectors $\bar{\vb}^e_i$, $\vc^e$ and the tensor $\mathcal{T}$, were introduced by equalities~\eqref{DfBr-a-LnFn}--\eqref{Df-a-LnFn} and \eqref{EqDfT&Ce}--\eqref{DfPhiAstI}, and $\bar{p}_e$ denotes the predator density at the equilibrium of the leading slow system.
\begin{rmrk}\label{RmOnAddVscsty}
\emph{In problem \eqref{EqMtrxEgnVlPrblm}, the coefficients $\delta_1\ge 0$ and $\delta_2\ge 0$ are responsible for the artificial diffusion. The case  $\delta_1=\delta_2=0$ corresponds exactly to the spectral stability problem for the linearized leading slow system \eqref{EqLnrPrdLdngSlw}--\eqref{EqLnrDrvrLdngSlw}. All assertions of this and the subsequent sections are valid for all sufficiently small nonnegative values of $\delta_1$ and $\delta_2$, including zero.  } 
\end{rmrk}

{For $k\neq 0$, each eigenvalue $\lambda=\lambda(\hx)$} of problem \eqref{EqMtrxEgnVlPrblm} has a complex conjugate pair, $\lambda^\ast=\lambda(-\hx)$. Each such pair generates a two-dimensional subspace of real-valued solutions to the linearized system, whose elements have the form:
\beear
\mathrm{Re}\,(\hat{a}(\hx)\exp(\lambda(\hx)t+i\hx\cdot x)).
\eear
Hence, any real eigenvalue, if it exists, has multiplicity of at least two, with the exception of the eigenmode corresponding to $k=0$. In particular, the dimension of a monotone stratum cannot be maximal, despite the realness of the original linearized system.

Spatially homogeneous modes (corresponding to $k=0$) are stable under  the conditions
$a_{{11}}+a_{{22}}<0$ and  $a_{{11}}a_{{22}}-a_{{12}}a_{{21}}>0$.  Additionally, we assume that $a_{12}>0,\ a_{21}<0$, since 
the prey (predator) density is naturally regarded as a factor enhancing (suppressing) the reproduction of the predator (prey).  

The explicit computation of coefficients $a_{ij}$ and $b_{i}$ in full generality appears to be too cumbersome. For simplicity, {we assume  $p$-linearity regarding }the kinetics of system \eqref{EqPrdTrns}--\eqref{EqSgnl}.  Then the coefficients $a_{ij}$ are determined by equalities \eqref{Df-a-LnFn}, where $p_0=p_e=\bar{p}_e$, $q_0=q_e=\bar{q}_e$,  $p_e>0$, and $q_e>0$ are the equilibrium species densities in the original system \eqref{EqPrdTrns}--\eqref{EqSgnl}; then, a direct check shows that $a_{11}=0$.  Moreover, $b_1=b_2=0$ (\emph{cf.} Remarks~\ref{RmOnQEqVsEq} and \ref{RmOnP-Lnr-1}). 

Thus,  the external signal controls only the coefficients $c$ and $\check{\chi}$. They also depend on the point of the unit sphere in the space of wave vectors, but for now, we shall regard them as additional independent parameters and assume that $\check{\chi}>0$. This inequality has not been proved in the general case, but it holds in all explicit examples at our disposal.

So, in the remaining part of the article, we assume
\beear
\check{\chi}> 0,\ \delta_1\ge 0,\ \delta_2\ge 0,\ \kappa>0, \nu>0,\ a_{12}>0,\ a_{21}<0,\ a_{22}<0,\ b_1=b_2=a_{11}=0.
\label{EqPrmDomn}
\eear
These inequalities, in particular, entail the stability of the homogeneous mode, and also exclude monotone instability even for $c=0$. Although matrix \eqref{EqMtrxEgnVlPrblm} is real for $c=0$,  its determinant is positive. 
 
We fix the kinetic terms of system \eqref{EqPrdTrns}--\eqref{EqSgnl} consistent with assumptions \eqref{EqPrmDomn} and \eqref{ExstCoExst}, and all the other conditions adopted above. Thereby we also fix the coexistence equilibrium (which coincides with the quasi-equilibrium by virtue of the assumed $p$-linearity) and the coefficients $a_{ij}$. In particular, the Lotka--Volterra kinetics with a logistic law of prey reproduction and under the condition of coexistence of species is suitable. In this case (with the normalization stated above), $f= \gamma q-\beta$, $g=1-p-q$,  where $\gamma>\beta>0$.

Consider the tuples of values of parameters  $\delta_1,\delta_2,\kappa,\nu$ ordered accordingly in this sequence. In this set, define the subset $\mathcal{A}$ uniting those tuples that consist of the parameters values obeying conditions~\eqref{EqPrmDomn}. Additionally, consider the extended tuples $\widetilde{\Pi}\byd (\Pi,c,\check{\chi},k)$, where $\Pi\in \mathcal{A}$, and define $\widetilde{\mathcal{A}}=\mathcal{A}\times\mathbb{R}\times (0,\infty)^2$. The compactness of a set of the parameter tuples (extended or not) is understood in the sense of the standard topology of $\mathbb{R}^m$.

According to one of the theorems of matrix theory \cite{GntMtrx}, ch. XV, given a \emph{complex} polynomial of order $N$, there is an algorithmic procedure to construct a $2N\times 2N$ \emph{real} matrix such that the number of sign changes in the chain of its main minors of  orders $0,2,4,\ldots,2N$ is equal to the number of roots of this polynomial in the open right semi-plane. (The zero-order minor is equal to unity by definition.) Hence,  given wave vector $\hx$, the number of corresponding unstable modes is equal to the number of sign changes in the chain of Hankel minors 
$
1,\Delta_2,\Delta_4,\Delta_6
$
associated to the characteristic polynomial of matrix \eqref{EqMtrxEgnVlPrblm}, where the subscripts are equal to the orders of the Hankel minors. Hence, the neutral strata of maximal dimension (necessarily oscillatory, as proved) consist of nondegenerate solutions to exactly one of the equations $\Delta_i(\widetilde{\Pi})=0$, $\widetilde{\Pi}\in \widetilde{\mathcal{A}}$, where the determinants $\Delta_i$, $i=2,4,6$ are real polynomials in the members of the tuple $\widetilde{\Pi}$.  They are expressed explicitly;\emph{ cf.} Appendix C. It is noteworthy that $\Delta_2>0$ everywhere on $\widetilde{\mathcal{A}}$, while $\Delta_4$ and $\Delta_6$ contain only even powers of the parameter $c$.  
\begin{prpsn}\label{PrpOnChi>0}
 \emph{\textsf{ Let $\mathcal{C}\subset \mathcal{A}$ be a compact set; given $c_*\in[0,\infty)$, set
  \beq
  \widetilde{\mathcal{C}}\byd\mathcal{C}\times [-c_*,c_*]\times[0,\chi_*]\times (0,\infty).
  \eeq
  Then 
  (i)  set $\{\Delta_4 =0\}$ is the graph of  rational function  $\check{\chi}={\chi}_4(\Pi,c,k)=\chi_{40}(\Pi,k)+c^2\chi_{42}(\Pi,k)$, defined everywhere on the set $\{(\Pi,c,k)\in \mathcal{A}\times \mathbb{R}\times(0,\infty)\}$, where the function $\chi_{40}$ is everywhere positive, and the function $\chi_{42}$ is nonnegative. Moreover,    
  $
  \forall\ \Pi\in\mathcal{A},\ \exists\, !\ k=k_4(\Pi)>0:
  $
  $
  \inf\limits_{k>0}\chi_{40}(\Pi,c,k)=\chi_{40}(\Pi,k_4(\Pi)).
  $
  \\
  (ii)
  $
  \inf\limits_{\{\Delta_4 =0\}\cap \widetilde{\mathcal{C}}}\check{\chi}=\inf\limits_{\{\Delta_4 =0\}\cap\{c=0\}\cap \widetilde{\mathcal{C}}}\check{\chi}=\chi_{40}(\Pi,k_4(\Pi))\byd \chi^*_{4}(\mathcal{C})> 0,
  $
  and   it is attained on a compact subset of the set $\widetilde{C}\cap\{c=0\}$.
  \\
(iii)   {either} $\inf\limits_{\{\Delta_6=0\}\cap \widetilde{\mathcal{C}}}\check{\chi}\byd \chi_6^*(\mathcal{C},c_*)>0$ {or}  $\inf\limits_{\widetilde{\mathcal{C}}}\Delta_6 >0$ provided that $c_*>0$ is small sufficiently.
\\
(iv) $\forall\ \mathcal{C}$, $\exists\ c_*>0:$ $\min(\chi^*_{4}(\mathcal{C}),\chi^*_{6}(\mathcal{C},c_*))>0$,\ {and}\ $\forall\ \Pi\in\mathcal{C}$ and  $\forall\ c\in(-c_*,c_*)$,  
  the equilibrium is stable for $0\le \chi <\chi^*_{cr}(\mathcal{C},c_*)\byd\min(\chi^*_{4}(\mathcal{C}),\chi^*_{6}(\mathcal{C},c_*))$.}}
\end{prpsn}
\noindent\textbf{Proof.} Assertion (i) follows by a direct check,  and (ii) follows from (i).  We pass to assertion (iii). From the explicit expression of the determinant $\Delta_6$, it follows that for a sufficiently small $c_*$, $|c|<c_*$ and for $\Pi\in\mathcal{A}$, 
$
\exists\, !\,m\in N:\ \lim\limits_{k\to\infty} k^{-m}\Delta_6(\Pi,c,k)=a(\Pi,c)\ \text{and}\quad \inf_{\Pi\in\mathcal{C},\,|c|<c_*}a(\Pi,c)>0.
$
This observation entails compactness for the set $\{\Delta_6=0\}\cap \widetilde{\mathcal{C}}$. Let it be nonempty, and then it contains a tuple $\widetilde{\Pi}^*_6$ such that its projection onto the $\check{\chi}$-axis (along the hyperplane of the other parameters) coincides with $\chi^*_6$.  The positivity of $\chi^*_6$ follows from the positivity of the polynomial $\Delta_6$ everywhere on the projection of the set $\widetilde{\mathcal{C}}$ along the $\check{\chi}$-axis   onto the hyperplane of the other parameters.  Item (iv) is a consequence of (iii) and (ii).
$\square$
\begin{rmrk}\label{RmOnNonDffsvLmt}
\emph{The determinant $\Delta_4$ does not depend on $c$ for $\delta_1=\delta_2=0$, that is, in this particular case, $\chi_{42}=0$ and $\chi_4=\chi_{40}$ does not depend on $c$. The noted feature does not affect the conclusions of this and the next sections. It is also worth noting that the coefficient $a$ introduced in the proof of Proposition~\ref{PrpOnChi>0} can indeed become negative for sufficiently large $c$. In that case, $\chi^*_6=0$ and normal modes with wave vectors filling unbounded subdomains lying outside some ball may turn out to be unstable. However, the inequalities $\delta_1>0$ and $\delta_2>0$ entail $a>0$, so that this short-wave instability is probably suppressed for arbitrarily small diffusions of the predator and the cue.}
\end{rmrk}
The determinant $\Delta_6$ is a third-order polynomial in $\check{\chi}$, positive for $\check{\chi}=0$ and for $\check{\chi}\to\infty$ (for an arbitrarily fixed collection of the remaining parameters).  Thus, either there are no positive zeros, or there are exactly two of them, counting multiplicities. Moreover, simple positive zeros (if any) determine the functions $\check{\chi}=\chi_{6}^\pm$, $\chi_6^+>\chi_6^-\ge 0$, on certain domains contained in the projection of the set $\widetilde{\mathcal{A}}$ along the $\chi$-axis onto the hyperplane of the other parameters, and $\chi_6^+=\chi_6^-$, or $\chi_6^-=0$, on the boundary of every such  domain. 

A multiple root $\chi_6^+=\chi_6^-$ occurs when $c=0$, and here the matrix of problem \eqref{EqMtrxEgnVlPrblm} is real.
In more detail: If $c=0$, then the determinant $\Delta_2$ is a divisor of $\Delta_4$, and the square of their ratio, $(\Delta_4/\Delta_2)^2$, is a divisor of $\Delta_6$. Moreover, $(\Delta_6\Delta^2_2/\Delta^2_4)$ is an everywhere positive polynomial. Hence, ${\chi}_{40}$ is a double root of the equation $\Delta_6|_{c=0}=0$. This is an expected effect, since when a pair of complex conjugate unstable eigenvalues of a real matrix arises, the distribution of signs along the Hankel chain must change as follows:
\beear
++++\,\mapsto\,++-+.
\eear 
It is precisely such a change of sign that occurs when the parameter $\check{\chi}$ passes the threshold value, ${\chi}_{40}$, by virtue of the divisibility and multiplicity indicated above. Moreover, the complex conjugate pair of neutral modes is double, in view of the mirror symmetry in $\hx$. At the same time, the simultaneous fulfillment of the conditions $c=0$ and $\check{\chi}<\chi^*_{4}$ entails the stability of the equilibrium.
 
The substitution $\check{\chi}={\chi}_{40}+\eps c$ transforms the equation $\Delta_6=0$ to the form $c^2\Delta_6^{\sharp}(c,\eps)=0$, where
\beq
\Delta_6^{\sharp}=-\Delta_{602}+\Delta_{620}\eps^2+c(\Delta_{630}\eps^3+\Delta_{612}\eps +c(\Delta_{622}\eps^2+\Delta_{604}+c(\Delta_{614}\eps
 +c \Delta_{606}))),
\eeq
 {and} $\Delta_{602}>0,\ \Delta_{620}>0$ for each collection of the remaining parameters. Hence, the equation $\Delta_6^{\sharp}(c,\eps)=0$ has solutions  $(0,\pm\sqrt{\Delta_{602}/\Delta_{620} })$ satisfying the conditions of the implicit function theorem with respect to the expression $\eps=\eps(c)$. Thus, there exist a couple of  functions $\eps^\pm(c)$  analytic for small $c$, and such that  their graphs pass through the points $(0,\pm\sqrt{\Delta_{602}/\Delta_{620} })$, and the equation $\Delta_6=0$ for small $c\neq 0$ has two distinct solutions ${\chi}_{6}^\pm(c)={\chi}_{40}+\eps^\pm(c)c$. Moreover $\chi_6^{\pm}(c)={\chi}_{40} \pm c\sqrt{\Delta_{602}/\Delta_{620} }+O(c^2)$, $c\to 0$.

Thus, the parts of the oscillatory strata of maximal dimension adjacent to the hyperplane $c=0$ are the graphs of the functions $\check{\chi}=\chi_6^\pm$ for $c\neq 0$.   The intersection with the plane $c=0$ of the union of the closures of these parts is a separate stratum of lower dimension lying on the graph of the function $\check{\chi}=\chi_{40}$.  We introduce the functions $\chi^-_{nt}=\chi_6^\pm$ for $\mp c>0$ and  $\check{\chi}=\chi^+_{nt}$, $\chi^+_{nt}=\chi^\pm$ for $\pm c>0$. In view of what has been said, for a suitable smallness of $c\neq 0$, we have $\chi^-_{nt}<\chi^+_{nt}$, and the distribution of signs along the Hankel chain with a gradual increase of $\check{\chi}$ changes as follows:
\beq
+\,+\,+\,+\ \stackrel{\check{\chi}={\chi}_{nt}^-}{\mapsto}\ +\,+\,+\,-\ \stackrel{\check{\chi}={\chi}_{4}}{\mapsto}\ +\,+\,-\,-\
\stackrel{\check{\chi}={\chi}_{nt}^+}{\mapsto}\ +\,+\,-\,+.
\eeq
According to this diagram, for an arbitrarily fixed nonzero wave vector $\hat{x}$ and for sufficiently small $c\neq 0$, there are no unstable modes for $\check{\chi}<{\chi}_{nt}^-$, there exists exactly one unstable mode for ${\chi}_{nt}^-<\check{\chi}<{\chi}_{nt}^+$, there exist exactly two unstable modes for ${\chi}_{nt}^+<\check{\chi}$, and nothing happens when the value ${\chi}_{4}$ is passed through. 

In fact, there are smooth extensions for functions $\chi_{nt}^\pm$ to a wider interval of parameter $c$ that keep parameterising the null set of minor $\Delta_6$ and preserve inequality ${\chi}_{nt}^-<{\chi}_4<{\chi}_{nt}^+$. To see this, observe that $\Delta_6|_{\check{\chi}=\chi_4}<0$ for every $c\neq 0$  and  for every collection of the other parameters. Hence,  equation $\Delta_6=0$ to be solved in $\check{\chi}$ has one positive root to the left of $\chi_{4}$ and another root to the right for every collection of the other parameters such that $c\neq 0$.

Thus, the above analysis and Proposition~\ref{PrpOnChi>0} entails the following: 
\begin{prpsn}\label{PrpOnChiCr}
\emph{\textsf{Under the conditions of Proposition~\ref{PrpOnChi>0}, (i) $\chi_{cr}(\mathcal{C},0)=\chi^*_{4}(\mathcal{C})$; (ii) $\forall\,\mathcal{C},\ \exists\, c_{*}>0:$ $\{\Delta_6=0\}\cap\widetilde{\mathcal{C}}\neq \varnothing$ and  $0<\chi_{cr}(\mathcal{C},c_{*})=\chi^*_6(\mathcal{C},c_{*})<\chi^*_{4}(\mathcal{C})$.}}
\end{prpsn}  
\noindent\textbf{Proof.} Indeed, $\exists\ \Pi_{4}\in \mathcal{C}:\ \chi_{40}(\Pi_{4},k_4(\Pi_4))=\chi^*_{4}(\mathcal{C})$, and at the same time, $\chi_{40}(\Pi_{4},k_4(\Pi_4))>\chi^{-}_{nt}(\Pi_{4},c,k_{4}(\Pi_4))\ge \chi_{cr}(\mathcal{C},c_*)$ $\forall\,c:\,|c|<c_{*}$, if $c_{*}$ is sufficiently small. $\square$ 
\subsection{Effect of the external signal}
\label{SscEffctSgnl}
\noindent
In this section, we give examples of inducing stability and instability of the coexistence equilibria by an external signal. Let $k>0$ and $\hat{y}\in \mathbb{R}^n:\,|\hat{y}|=1$, and consider the normal mode associated with the wave vector $\hat{x}=k\hat{y}$.  The corresponding problem \eqref{EqMtrxEgnVlPrblm} depends on the sensing coefficient and the drift coefficient, $\check{\chi}=\hat{y}\mathcal{T}\hat{y}$ and $c=\hat{y}\cdot\vc^e$. We set $\chi_*\byd \sup \hat{y}\mathcal{T}\hat{y}\ge \check{\chi}$, $c_*=|\vc_e|\ge |c|$.

Without an external signal, the drift vanishes and the sensing tensor becomes trivial---that is, $\bar{\mathrm{V}}(\bar{\vw})=0\ \forall\,\bar{\vw}$, and therefore $\mathcal{T}=E$, $\vc^e=0$, and hence, $c_*=c=0$, ${\chi}_*=\check{\chi}=1$; see \eqref{DfF&V} and~\eqref{EqDfT&Ce}. We set $\mathcal{A}^0_{cr}\byd\{\Pi\in \mathcal{A}:\,\inf_{k>0}\chi_{40}(\Pi,k)=1\}$. This set is nonempty generically, and also in the case of Lotka--Volterra kinetics, for example. Further, let $\mathcal{C}^0_{cr}\subset \mathcal{A}_{cr}^0$ be a compact set, and then $\chi^*_{4}(\mathcal{C}^0_{cr})=1=\chi_{cr}(\mathcal{C}^0_{cr},0)$; see Proposition~\ref{PrpOnChiCr}. Thus, without a signal, set $\mathcal{A}^0_{cr}$ is a subset of the critical set of the equilibrium separating the regions of its stability and instability; see Section~\ref{SscLnrStbDcsn}.   

Suppose an external signal is supplied such that $c_*=0$, for example, a stationary one (see Section~\ref{SsscStSgnl&Lnrzn}), and let $\chi_*>1$. Then $\forall\, \Pi\in\mathcal{A}^0_{cr},\ \exists\, \hat{y}$ and $k>0$, such that the wave vector $\hat x=k\hat{y}$  corresponds to an unstable mode, that is, the quasi-equilibrium is unstable.   If, however, $\chi_*<1$, then the quasi-equilibrium is stable.

%Let us agree that the collection of parameters $\Pi\in \mathcal{A}$ belongs to the set $\{\inf_{k}\chi_{40}=1\}\neq \varnothing$. Let us note in passing that %${\chi}_{cr}={\chi}^*_{4}$ for $c=0$ (see proposition~\ref{PrpOnChiCr}). We denote by $k^*_{40}$ the minimum point of the function ${\chi}_{40}$: %${\chi}_{40}(k^*_{40})={\chi}^*_{40}=1$.

In the case of a generic stationary signal, all eigenvalues of the matrix of the corresponding tensor $\mathcal{T}$ belong to the interval (0,1)  (in view of Lemma~\ref{LmmTrnspTnsr} and Remark~\ref{RmOnMinEgnTrTn}). Then 
$
{\chi}_*<1={\chi}_{cr}={\chi}^*_{40}.
$ 
Hence, the quasi-equilibrium for a generic stationary signal is stable, while the corresponding equilibrium without a signal is critical in the sense of stability. Thus, a stationary signal has, generally speaking, a stabilizing effect.  It manifests itself brightly if the external signal satisfies condition \eqref{EqH=SumHiStd}. Then the sensing tensor is diagonal, and its diagonal elements tend to zero   exponentially as $A_0\to\infty$, where $A_0$ is the characteristic value of the effective amplitudes of the additive components of the signal, $h_j=h_j(\xi_j)$;  see formula \eqref{EqTrTnsrSttnr} and Remark~\ref{RmOnLnrExpSprssng}.

Now, let the external signal be defined by formula \eqref{H=TrWv}. Then it is an additive superposition of plane traveling waves propagating along the coordinate axes with velocities $\omega_j$, $j=1...n$. Let their effective amplitudes be tied to the characteristic value $A_0$.   Here formula \eqref{EqRsdDrftTrWv} determines the residual drift, and formulae \eqref{EqTrTnsTrWv} determine the sensing tensor, which is moreover diagonal.   We fix $\eps>0$ and choose $A_0>0$ and $\omega_*>0$ so that $\chi_*<1-\eps$ when $\sup_j\omega_j<\omega_*$.   The possibility of such a choice follows from Remark~\ref{RmOnInrvtng}, which also implies that there is a threshold $\omega^*>0$ such that the inequality $\omega_j>\omega^*$ for at least one $j=1...n$ entails the inequality $\chi_*>1$. Let, for definiteness, $\omega_n>\omega^*$ and  $\hat{y}=e_n$, and then $\vc^e\hat{\vy}>0=c_n=c$. Without loss of generality, we assume  $c>0$ and $e_n\mathcal{T}e_n=\chi_*$. 

Further, in view of Remark~\ref{RmOnC_jOm_jToInfty}, $c_j=O(\omega^{-1}_j)$, $j=1...n$, both as $\omega_n\to+0$ and as  $\omega_n\to\infty$. 
We choose $\omega_*$ and $\omega^*$ so that for $\inf_j\omega_j>\omega^*$ or for  $\sup_j\omega_j<\omega_*$, the quantity $c_*$ is so small that both Propositions~\ref{PrpOnChi>0} and~\ref{PrpOnChiCr} hold. Then, having increased, if necessary, $\omega_*$, we shall have $0\neq |c|=|c_n|<c_*$ and
$
\chi_{cr}(\mathcal{C}^0_{cr},c_*)=\inf_{k>0,|c|<c_*,\Pi\in \mathcal{C}^0_{cr}}\chi_{nt}^-(\Pi,c,k)<\chi^*_{4}(\mathcal{C}^0_{cr})=1<\chi_*\,\ \forall\ \Pi\in \mathcal{C}^0_{cr}.
$

At the same time, having decreased, if necessary, $\omega_*$, $\forall\ \omega:\, \sup_j\omega_j<\omega_*$, we shall have
\beq 
0\neq |c|=|c|<c_*\ \text{and}\ \chi^*_{4}(\mathcal{C}^0_{cr})=1>
\chi_{cr}(\mathcal{C}^0_{cr},c_*)=\inf_{k>0,|c|<c_*,\Pi\in \mathcal{C}^0_{cr}}\chi_{nt}^-(\Pi,c,k)>1-\eps=\chi_*.
\eeq
 Thus, we arrive at the following assertion. Choose a tuple of parameters $\Pi$ on  the critical set of the coexistence equilibrium of the system without an external signal, and let the signal have the form~\eqref{H=TrWv}.  The characteristic amplitude of such a signal can be chosen so that the coexistence quasi-equilibrium is stable when all the waves constituting the signal propagate with sufficiently small velocities and unstable when the propagation velocity of at least one of them is sufficiently large. Moreover, the threshold propagation velocities can be chosen the same for each compact subset of the critical set of the coexistence equilibrium of the system without an external signal. 

Let us illustrate what has been said with a concrete example; see Figure~\ref{FgChiAstVsChiNtMns}. Consider a system with two spatial variables and with Lotka--Volterra kinetics, so that $x=(x_1,x_2)$, $\xi=(\xi_1,\xi_2)$, $\eta=(\eta_1,\eta_2)$, $f=\gamma q-\beta$, $\beta<\gamma$, $g=1-q-p$. Then $\exists\, \kappa=\kappa^0_{cr}(\gamma,\beta,\nu)>0:$ $\mathcal{A}^0_{cr}\supset\{\delta_1=0,\delta_2=0,\kappa^0_{cr}(\gamma,\beta,\nu),\nu\}$. We set $k^0_{cr}=k_4(\Pi^0_{cr})$, where $\Pi^0_{cr}= (0,0,\gamma,\beta,\kappa^0_{cr},\nu)$.   We fix $\gamma=4/3, \beta=1/3$. For $\mathcal{C}^0_{cr}$, we choose one point of the parameter space with coordinates $\delta_1=0,\delta_2=0,\nu=1,\kappa=\kappa^0_{cr}(4/3,1/3,1)$. This choice uniquely defines the magnitude of  $k_{cr}^0$. Thus, without a signal, the parameters of the system are on the critical set of the equilibrium. For a stationary signal, the inequality $\chi_*>1$ ($\chi_*<1$) entails instability (stability) of the quasi-equilibrium.  We choose the signal in the form \eqref{H=TrWv}, where $n=2$, $h_1(\eta)=h_2(\eta)=\sin\eta$, and let $\omega_1=0$, $\omega_2=\omega>0$. For $\omega=0$, we have $\mathcal{T}_{11}=\mathcal{T}_{22}=I_0(1)^{-2}<0.7<1$ ($I_0$ is the modified Bessel function of the first kind), so that the quasi-equilibrium is stable. Let $\omega\neq 0$,  consider the mode with wave vector   $k^0_{cr}e_2$, where  $e_2$ is the unit vector of the axis $Ox_2$. Now the quantity $\chi_{nt}^-$ is expressed as a function of $c$ only, and $c$ is a function of $\omega$ only. The quantity $\mathcal{T}_{22}$ also depends only on $\omega$. The graphs of the functions $\chi_{nt}^-(\omega)$ and $\mathcal{T}_{22}(\omega)$ are shown in Figure~\ref{FgChiAstVsChiNtMns} in red and blue, respectively. The mode with wave vector $k^0_{cr}e_2$ is unstable (stable) for the values of $\omega$ lying to the right (left) of the projection of the intersection point of these graphs onto the abscissa axis. It is worth noticing that the graphs' behavior {for relatively small and relatively high values of the signal wave speed, $\omega$,} qualitatively matches what we have been saying above. %\newpage
{\begin{figure}[H]
\hspace{20pt}\includegraphics[scale=0.3]{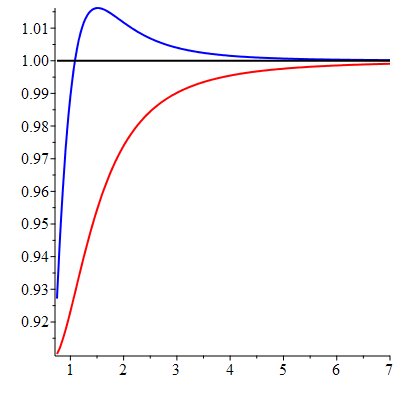}\hspace{20pt}\includegraphics[scale=0.3]{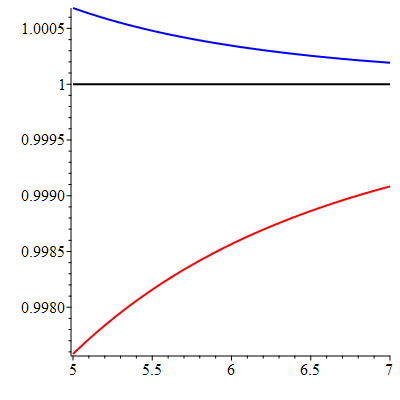}\hspace{20pt}\includegraphics[scale=0.3]{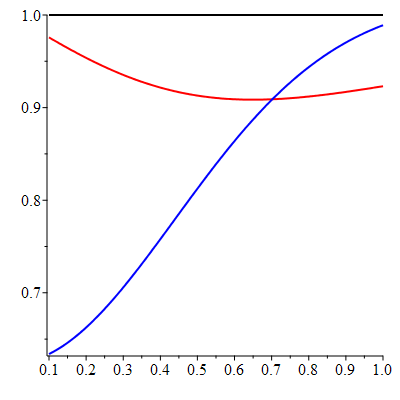}
\caption{Three fragments of the graphs $\chi_{nt}^-=\chi_{nt}^-(\omega)$ (red) and $\mathcal{T}_{22}=\mathcal{T}_{22}(\omega)$ (blue). The black line corresponds to the critical value $\chi_*$ for a stationary signal.}
\label{FgChiAstVsChiNtMns}
\end{figure}}
\section{Discussion and conclusions}\label{ScCncl}
\noindent
We have addressed a spatial ``predator--prey" system in multiple dimensions with indirect Patlak--Keller--Segel taxis driven by the cue field induced by the prey while assuming the cue field induction to be additionally influenced by a short-wave external signal. For this system, we have proposed a procedure for constructing asymptotics of arbitrarily high approximation order while assuming no special forms of the signal or of solutions or system kinetics. So, we have generalized the previous result \cite{AM1}  substantially.

The leading term of this asymptotics has the form of a {short-wave parcel}.  Its amplitude undergoes slow modulation, which, in turn, is governed by the leading slow system, written in explicit form. The latter, in comparison with the original system without a signal, exhibits two main features caused only by the external signal: predator drift and modification of the kinetics of the system. By  "drift" we mean the deviation of the advection velocity of predators from the cue gradient. We present a particular class of signals that allows for expressing the drift explicitly. The interesting question for further investigations is whether the long-term evolution of the aforementioned package brings it into  a soliton-like form.   

The kinetics of the system is not modified if it is $p$-linear—that is, the kinetic terms $pf(p,q)$ and $qg(p,q)$ are linear in $p$. Interestingly,  this type of kinetics is widely recognized in population dynamics; \emph{cf.} article \cite{TtnTrFn} and monograph \cite{ArdG} for more information on the subject. 
%At the same time, this article also gives arguments in favor of kinetics depending on the ratio of predator and prey densities, which is not p-linear.

Since the original system without a signal possesses translational invariance, the leading slow system inherits it under the condition that the external signal does not undergo slow modulation. This case allows one to identify quasi-equilibria—that is, short-wave solutions not subject to slow modulation—with homogeneous equilibria of the leading slow system and to consider the induced stability by analogy with the theory of the inverted pendulum of P. Kapitza.  The external signal influences the stability of these equilibria through three types of control parameters depending only on the signal, namely, the components of the residual drift velocity  of the predators, $\vc^e$, the components of the sensing tensor, $\mathcal{T}$, and the components of the two cue drift velocity vectors, $\vb_i^e$, $i=1,2$, arising upon the linearization of the modified kinetics.  {Furthemore,}  $\vc^e=0$ in the case of stationary external signals and {$\vb_i^e=0$}, $i=1,2$, in the case of $p$-linear kinetics.

The sensing tensor measures the anisotropy of the sensing of predators to the cue. A typical stationary external signal generates a sensing tensor indicating that the sensing is to some extent suppressed in any direction. Moreover, there exists a class of stationary signals for which the sensing tensor is computed explicitly, and on this basis, we {have discovered} the  examples of exponential decay of the sensing along one or several spatial coordinate axes as the intensity of the external signal grows. At the same time, a nonstationary external signal can enhance the sensing and tactic transport in certain directions. In particular, we have observed the inversion of the effect of the external signal  having   the form of a superposition of waves of fixed profiles traveling along different coordinate axes. Namely, such a signal destabilizes the quasi-equilibrium when the propagation velocity of at least one wave is sufficiently large and stabilizes it when the velocities of all waves are sufficiently small. This destabilization can be described as a kind of blurring of the boundary between the stability and instability regions existing in the parameter space of the system without a signal. We are not aware of examples of a “predator--prey” system with taxis, direct or indirect, that admits destabilization of the coexistence equilibrium solely due to the external signal, in the sense that the equilibrium becomes stable for any set of parameters of the system when the external signal is off.

The stability of an individual mode depends substantially on the direction of the wave vector $\hat{\vx}$. Among the modes corresponding to the same $|\hat{\vx}|$, the choice of the most dangerous one (in the sense of instability) depends on the relations between $\vc^e\cdot \hat{\vx}$ and $\hat{\vx}\cdot\mathcal{T}\hat{\vx}$. In view of this, questions naturally arise concerning the control and optimization of the residual drift and the sensing tensor, and the same can be asked about the nonlinear drift, ${\mathrm{\bar{V}}}$. This  remains for future research.

Although the leading term of the asymptotics gives us a summary of the most essential information about the solution, higher-order terms (\emph{cf.} Appendix B) may play an important role in justifying the asymptotic approximation. We are optimistic about the possibility of solving this issue in the future, since there is considerable positive experience in similar problems \cite{LevCnv1993,LevIzvRAN06,LevAA14,LevSmz23,Allr-2,MrchKhrs}.

In conclusion, we note that the homogenization of the response to a short-wave external signal proposed in this article reduces to solving linear parabolic equations to evaluate the drift, which is anyway feasible using numeric techniques.  Hence, there exists a relatively simple possibility of predicting the behavior of predator--prey-type systems with taxis in short-wave external fields and controlling them by means of artificially created external signals.
\vspace{10mm}\\
{\sc {\footnotesize {\bf Acknowledgments.} The authors did not receive funding from any organization in support of the study presented herein. At the same time, the authors are grateful to the Southern Federal University for the opportunity  of doing this research. For the same reason, A. Morgulis is grateful to the Southern Mathematical Institute of the Vladikavkaz Scientific Center of the Russian Academy of Sciences.} }
 
\section*{Appendices}
\addcontentsline{toc}{appendix}{Appendices}
\appendix
\counterwithin{equation}{section}
\section{Appendix A. Details of the proofs.}
%\addcontentsline{toc}{subsection}{ Details of the proofs.} 
\label{ScApp1}\noindent
\subsection{Differentiability}
%\addcontentsline{toc}{subsubsection}{Proof of differentiability}
\noindent
 Recall  the differential operator $\mathcal{L}=\mathcal{L}(\vw,\eps):\mathrm{L}_2(\mathbb{T}_{\xi,\tau}^{n+1})\to \mathrm{L}_2(\mathbb{T}_{\xi,\tau}^{n+1})$  and its right inverse $\mathcal{L}^{-1}( {\vw},\eps):\widetilde{\mathrm{L}}_2(\mathbb{T}_{\xi,\tau}^{n+1})\to \widetilde{\mathrm{L}}_2(\mathbb{T}_{\xi,\tau}^{n+1})$ introduced by Definition~\ref{DfL(w)} and Remark~\ref{RmOnPhi}, together with the mapping $\Phi:\vw\to  \Phi(\vw,\eps)$. Here the wide-tilde notation specifies  $\mathrm{range}(I-M)$-subspace, where $M$ is the averaging operator over the torus.
\begin{lemma}\label{LmmDffPhiL2}
\emph{\textsf{Given $\vw\in \mathrm{C}^\infty(\mathbb{T}_{\xi,\tau}^{n+1},\mathbb{R}^n)$ and $\eps>0$, the mapping $\check{\Phi}: y\mapsto \Phi(\vy+\vw,\eps)$, where $y\in \mathbb{R}^n$ and $\vy$ is the vector field on the torus that corresponds to $y$, is $\mathrm{L}_2$-differentiable at the origin and 
\beq
\left.d{\check{\Phi}}(v)\right.|_{y=0}=-v_i\astphi_i,\ i=1...n,\quad \astphi_i=\mathcal{L}^{-1}(\vw,\eps)\pr_{\xi_i}\check{\Phi}(0),
\eeq
where the summation convention is in effect.}}
\end{lemma}
\noindent\textbf{Proof.} Consider the operator family $y \to \mathcal{L}(\vy+\wtlnbl\tilde{s}_1,\mu_1)$ consisting of the operators specified by Definition~\ref{DfL(w)}. Given this family member, Lemma~\ref{LmmExst} describes  its spectral projector, $P$.   The well-known results in the perturbation theory for the linear operators imply that the operator-valued function, $P$,  is analytic in $y$ in some neighborhood
of the origin.  Then the conjugated function, $P^*$ is analytic too, and 
 its differential, ${{P}}^*_1:v\to {{P}}^*_1v$, satisfies the limit identity
$
\rho^{-1}({{P}}^*_1(\rho y)-{{P}}^*_1(0))\to {{P}}^*_1y,\ \rho\to 0,
$
relative to the uniform operator topology. Hence, by the definition of projector ${P}(y)$,  
\beear
&\rho^{-1}\langle \psi_1 (\check{\Phi}(\rho y)-\check{\Phi}(0))\rangle\langle \psi_2\rangle=\rho^{-1}\langle \psi_1 ({{P}}(\rho y)-{{P}}(0))\psi_2\rangle\to \langle \psi_1 ({P}_1y)\psi_2\rangle 
\ \forall\ \psi_1,\psi_2\in {\mathrm{L}_2}\Rightarrow &
\label{EqDffPhi}\\
 &\langle \psi_1 ({P}_1y)\psi_2\rangle=0\ \forall\ \psi_2:\ \langle \psi_2\rangle=0,
&\nonumber
\eear
and this, in turn, entails the following assertion:
\beq 
\exists\,\{\astphi_i, \ i=1...n\} \subset  {\mathrm{L}_2}:\langle \astphi_i\rangle=0,\ ({P}_1v)\psi=\langle \psi\rangle\astphi_i v_i,
 \eeq
 where the summation agreement is in effect. The constraints  imposed here  on functions $\astphi_i$ are necessary to obey the condition ${P}{P}_1+{P}_1{P}={P}_1$, which, in turn, follows from the  projector-valuedness of function ${P}$.
Moreover, by limit identity \eqref{EqDffPhi} (written for $\psi_2=1$), $\forall\,\delta>0$, $\exists\,\rho>0:$
 \beq
 \langle \psi_1 (\rho^{-1} (\check{\Phi}(\rho y)-\check{\Phi}(0))-\astphi_i y_i)\rangle\le \delta\|\psi_1\|, 
\eeq
and this is the $\mathrm{L}_2$-differentiability of the mapping $y \mapsto \astphi=\check{\Phi}(y)=P(y)1$ at the origin. For calculating the corresponding differential, let us note that   $\langle (P(y)1) {\check{\mathcal{L}}}^*(y)\psi\rangle=0\ \forall \psi$,  where ${\check{\mathcal{L}}}(y)=\mathcal{L}(\vw+\vy,\eps)$ and ${\check{\mathcal{L}}}(y)=\mathcal{L}^*(0)-\vy\wtlnbl$. Hence,
\beq
0=\left.d(y\mapsto \langle\check{\Phi}(y){\check{\mathcal{L}}}^*(y)\psi\rangle)\right|_{y=0}(v)=(\langle \astphi_i{\check{\mathcal{L}}}^*(0)\psi\rangle-\langle \astphi\psi_{\xi_i} \rangle)v_i,\ \text{where}\ \astphi=\check{\Phi}(0).
\eeq
These identities (valid for every test function ${\psi}$ (in the fast variables) and every $v\in \mathbb{R}^n$) give the weak formulations to the following equations:
\beear
\check{\mathcal{L}}(0)\astphi_i=-\astphi_{\xi_i},\ \langle \astphi_i\rangle=0,\quad i=1...n.
\label{DfPhiAstII}
\eear
Additional constrains, $\langle \astphi_i\rangle=0$,   arise once again from differentiating the constraint $\langle{{P}}\check{\Phi}(y) \rangle=1\ \forall\, y$. However, they get satisfied by the construction of the right inverse to $\mathcal{L}$; see Remark~\ref{RmOnPhi}. Hence, we arrive at what the lemma claims.
$\square$
\subsection{Proof of Lemma~\ref{LmmTrnspTnsr}}
%\addcontentsline{toc}{subsubsection}{Proof of Lemma~\ref{LmmTrnspTnsr}.} 
\noindent
Define the Hilbert space $\mathrm{L}_{2,\efr}$ as the vector space of $\mathrm{L}_{2}$-summable real vector fields on the torus $\mathbb{T}^{n}_{\xi}$, equipped with the metric
\beq
\|\vw\|^2_{\efr}=\langle\efr\vw^2\rangle. 
\eeq  
Let  $\bar{\vu}=(u_1,..,u_n)$ and $\bar{\vv}=(v_1,..,v_n)$ be a pair  of vector fields on the torus of fast variables, associated with constant vectors in $\mathbb{R}^n$, and $\mathcal{L}=\mathcal{L}(\wtlnbl \tilde{s},\mu_1)$. Since $\wtlnbl \tilde{s}= -\mu_1\efr \wtlnbl\efr^{-1}$, we have
$
\mathcal{L}u=- \mu_1\wtlnbl\cdot(\efr\wtlnbl(\efr^{-1}u)).
$ 
Further, 
\beq
u_jv_k\langle \efr\rangle\langle\astphi_j\tilde{s}_{\xi_k}\rangle= \mu_1\langle(\bar{\vv}\wtlnbl\tilde{s})\mathcal{L}^{-1}(\bar{\vu}\wtlnbl\efr)\rangle=
% \mu_1\langle\efr^{-1}(\bar{\vv}\wtlnbl\efr)\mathcal{L}^{-1}(\bar{\vu}\wtlnbl\efr)\rangle=
- \mu_1\langle\bar{\vv}\efr\wtlnbl\efr^{-1}\mathcal{L}^{-1}\wtlnbl\cdot(\efr\bar{\vu})\rangle.
\eeq 
We continue the right-hand side of the last equality in a natural way to a bilinear form on the space $\mathrm{L}_{2,\efr}$, which we transform as follows (we omit the coefficient   $\mu_1$):
\beear \label{BlnrChn}
-\langle{\vv}\efr\wtlnbl\efr^{-1}\mathcal{L}^{-1}\wtlnbl\cdot(\efr{\vu})\rangle=
\langle\wtlnbl\cdot(\efr{\vv})\efr^{-1}\mathcal{L}^{-1}\wtlnbl\cdot(\efr{\vu})\rangle=
\langle\wtlnbl\cdot(\efr{\vv})\left((I-M)(\efr^{-1} \mathcal{L}^{-1})\right)\wtlnbl\cdot(\efr{\vu})\rangle.
\eear
Here it is useful to note that $\langle f,Mg\rangle=0$ if $\langle g\rangle\langle f\rangle=0$. Recall that the space  $\widetilde{\mathrm{L}_{2}}$ consists of square-summable scalar functions on the torus with zero mean and is equipped with the metric $|f|^2=\langle ff^*\rangle$.  {Define  the operator  $\mathcal{L}_{\efr}:\widetilde{\mathrm{L}_{2}}\to \widetilde{\mathrm{L}_{2}}$ by the mapping $\mathcal{L}_{\efr}:\eta\mapsto -\mu_1\wtlnbl\cdot(\efr\wtlnbl \eta)$.} It is symmetric and has a bounded right inverse $(I-M)\efr^{-1}\mathcal{L}^{-1}: \widetilde{\mathrm{L}_{2}}\to \widetilde{\mathrm{L}_{2}}$, which, hence,  is symmetric. Therefore the bilinear forms forming the chain \eqref{BlnrChn} are symmetric. Thus, there arises   a symmetric bounded operator $Q:\mathrm{L}_{2,\efr}\to \mathrm{L}_{2,\efr} $, 
$Q:\vu\mapsto - \mu_1\wtlnbl \left((I-M)(\efr^{-1} \mathcal{L}^{-1})\right)\wtlnbl\cdot(\efr{\vu})$, where the factor $(I-M)$ can be omitted.
Moreover, the operator $Q$ is a projector and an orthogonal projector (since it is symmetric). Indeed,
\beq
Q^2\vu= \mu_1^2\wtlnbl \efr^{-1} \mathcal{L}^{-1}\wtlnbl\cdot(\efr \wtlnbl (\efr^{-1} \mathcal{L}^{-1}\wtlnbl\cdot(\efr \vu)))=
- \mu_1\wtlnbl \efr^{-1} \mathcal{L}^{-1}\wtlnbl\cdot(\efr \vu)=Q\vu.
\eeq
Hence,  
\beq
u_ju_k\langle \efr\rangle\langle\astphi_j\tilde{s}_{1\xi_k}\rangle=(\bar{\vu},Q\bar{\vu})_{\efr}\ge 0
\quad \text{and}\quad u_ju_k\langle \efr\rangle\langle\astphi_j\tilde{s}_{1\xi_k}\rangle\le \|\bar{\vu}\|_{\efr}^2=\langle\efr\rangle \sum\limits_{i} u_i^2.
\eeq
Suppose that $Q\bar{\vu}=\bar{\vu}$. Then $\bar{\vu}=\wtlnbl\phi$, which is false for a function on the torus. This completes the proof, since $\bar{\vu}\mathcal{T}\bar{\vu}=\bar{\vu}^2-(\bar{\vu},Q\bar{\vu})_\efr\langle \efr\rangle^{-1}$.
$\square$
\begin{rmrk} \label{RmOnMinEgnTrTn1}
\emph{ Let the external signal $\tilde{h}=\tilde{h}(\xi)$ be invariant with respect to translations  along some subspace in $\mathbb{R}_\xi^n$. Then there exists vector $\bar{\vw}$ such that
$
\bar{\vw}\wtlnbl\efr=0=\wtlnbl\cdot(\efr\bar{\vw}),
$
and hence $Q\bar{\vw}=0$. The converse is also true. Hence, a generic stationary signal generates a sensing tensor such that each of its eigenvalues is strictly less than unity.}
\end{rmrk}
\section{Appendix B. Higher approximations}
%\addcontentsline{toc}{section}{Appendix B. Higher approximations} 
\label{ScApp2}%\stepcounter{section}
\noindent
Here we address the chain of equation arising for the terms of order $k\ge 3$ of expansion \eqref{AsmpExpnsn}. Bearing in mind what we already know about the terms of lowest orders, we write these equations as follows 
\begin{eqnarray}
& 
\mathcal{L}(\vw_0,\mu_1)p_{k-1}=\mu_1\Delta p_{k-3}-{(\pr_t-\mu_1\mathcal{D})}p_{k-2}
&
\nonumber\\
&
+\sum\limits_{j=0..k-2}(p_jf_{k-2-j}-{\nabla\cdot}(p_j\vw_{k-2-j})   -{{\wtlnbl}\cdot}(p_j\vw_{k-1-j})),
&
\label{EqPrdCfcnt}\\ 
&
\text{where}\quad \vw_{m}= (\nabla s_{m}+{\wtlnbl}{s}_{m+1}),\ m=0,1,\ldots,
&
\nonumber
\\
&{-\widetilde{\Delta}}{q}_{k}=\sum\limits_{j=0..k-2}q_jg_{k-2-j}-\left(\pr_\tau-\mathcal{D}\right)q_{k-1} -(\pr_t-\Delta)q_{k-2},&
\label{EqPryCfcnt}
\\
&
\mathcal{H}_{\mu_2} s_{k}={{\kappa}} q_{k-1}-\nu s_{k-1}-(\pr_t-\mu_2 \mathcal{D})s_{k-1} +\mu_2\Delta s_{k-2},
&
\label{EqDrvrCfcnt}
 \end{eqnarray}
 %where $\delta_{ik}$ is Kronecker's symbol.
where coefficients $f_k,g_k$ result from expanding the   reaction terms as follows  
\begin{eqnarray}&\label{DfFm-Gm}
  f(p,q)=\sum\limits_{m=0}f_m\delta^m,\  f_0=f(p_0,q_0),\quad  g(p,s)=\sum\limits_{m=0}g_m\delta^m,\  g_0=g(p_0,q_0).
&\end{eqnarray}
We have omitted  the explicit expressions  here for the exposition's compactness. 

In this appendix, we have to deal with the projection of various functions arising from the decompositions of the space $\mathrm{L}_2(\mathbb{T}^{n+1}_{\xi,\tau})$  due to the actions of projective operators $M$, $M^\tau$, $M^\xi$, $P$; see Section~\ref{ScSttng} for their definitions. For convenience, we set the notations for these as follows. Given a function $f$, define 
\begin{eqnarray*}
&
\bar{f}=\langle f \rangle=Mf,\ \widetilde{f}=(I-M)f,\quad \overline{f}^\xi=M^\xi f,\ \widetilde{f}^\xi=(I-M^\xi)f,\quad 
\overline{f}^\tau=M^\tau f,\ \widetilde{f}^\tau=(I-M^\tau)f,
&
\\
&
\widetilde{f}^\sharp=(M^\xi-M)f,\quad f^\bullet={P}f=\langle f\rangle\phi_*,\ f^\circ=(I-{P})f,
&
\end{eqnarray*}
so that $f=\bar{f}+\tilde{f}=f^\bullet+f^\circ,\quad f=\bar{f}+\widetilde{f}^\sharp +\widetilde{f}^\xi,\quad \tilde{f}=\widetilde{f}^\xi+\widetilde{f}^\sharp.$

{We recall that each equation in} the chain of Eqs \eqref{EqPrdCfcnt}--\eqref{EqDrvrCfcnt} has to be solved on {the} torus $\mathbb{T}_{\xi,\tau}^{n+1}$. For the subsequent analysis, it is substantial to see the triangle form of system \eqref{EqPrdCfcnt}--\eqref{EqDrvrCfcnt}. Indeed,  the system allows for isolating unknowns $s_k$ and $q_k$ from the third and second equations correspondingly and then isolating   $p_{k-1}$ from the first one, provided all these equations are compatible.  It turns out that the compatibility of every separate equation in system \eqref{EqPrdCfcnt}--\eqref{EqDrvrCfcnt} makes it compatible as a whole.

We restore the notation $\tilde{s}_1$ in accordance with Eq \eqref{Q1&S1}.  
For the compactness of the layout, we will use the notation $\mathrm{Op}_m$, $m=1,2,3,\dots$, for every operation (possibly differential or non-local) on only those coefficients $p_k,q_k,s_k$ that are indexed by $k\le m$. For instance,  
\beq
\sum\limits_{j=0..m}p_jf_{m-j}-a_1(p_{m},q_{m})=\sum\limits_{j=1..m-1}p_jf_{m-j}=Op_{m-1},
&\\& 
\sum\limits_{j=0..m}q_jg_{m-j}-a_{2}(p_{m},q_{m})=\sum\limits_{j=1..m-1}q_jg_{m-j}=Op_{m-1},
\eeq
where  linear functions $a_i$  are to be evaluated by formulae \eqref{Df-a-LnFn} using coefficients $p_0$ and $q_0$ taken from the leading term in the asymptotics. 
Besides, it follows from Eqs \eqref{EqPryCfcnt} and~\eqref{EqDrvrCfcnt} that
\beear
&
\tilde{{s}}_m=Op_{m-1},\quad \wqx_m=Op_{m-1},\ \wqta_m=Op_{m-1}, \tilde{q}_{m}=\wqx_m+\wqta_m=Op_{m-1}.
&
\label{TldSm=&TldQm=}
\eear
At the same time,
\beq
\pcr_{m}=\bar{p}\mathcal{L}^{-1}(\vw_0,\mu_1){\wtlnbl}\astphi\nabla \bar{s}_{m} +Op_{m-1}.
\label{PCRm=}
\eeq
Indeed, it is obvious that
\beear
&
\pcr_{m}=\mathcal{L}^{-1}(\vw_0,\mu_1)((\bar{p}{\wtlnbl}\astphi\nabla \bar{s}_{m}) + {\wtlnbl}\cdot(p_0{\wtlnbl}\tilde{s}_{m+1}))+Op_{m-1},
&
\label{TldPm}
\eear
and the second summand can also be identified with an $Op_{m-1}$ expression, since 
\beq
\tilde{s}_{m+1}=\mathcal{H}^{-1}_{\mu_2} ({{\kappa}} \tilde{q}_{m}-\nu \tilde{s}_{m}+\mu_2\Delta \tilde{s}_{m-1}-(\pr_t-\mu_2\mathcal{D})\tilde{s}_{m})=Op_{m-1}.
\eeq
Additionally, we define a pair of functions, $\bar{a}_i:\mathbb{R}^2\to \mathbb{R}$, $i=1,2$, and a pair of vector fields, $\vb_i$, $i=1,2$, by formulae \eqref{DfBr-a-LnFn}--\eqref{Df-a-LnFn} for $p_0=\bar{p}\astphi,q_0=\bar{q}$, where $\bar{p}$, $\bar{q}$ are the solutions to the leading slow system that determine the leading term in the asymptotics.
\begin{dfntn}\label{DfSnrSlwSst}
\emph{\textsf{By the slow system of order $\ell=1,2,\ldots$, we mean the system of equations for the unknown functions $\bar{p}_\ell,{\bar{q}}_\ell,\bar{s}_\ell$ in the slow variables $x,t$, which reads
 \begin{eqnarray}
& 
\bar{p}_{\ell,t}+{ \nabla\,\cdot\,}(\bar{p}_{\ell}\bar{\vv}+\bar{p}\nabla\bar{s}_{\ell})
-\bar{\vb}_1\nabla \bar{s}_\ell-\bar{a}_1(\bar{p}_{\ell},\bar{q}_{\ell})=
\mu_1\Delta\bar{p}_{\ell-1}+ \left\langle a_1(r_\ell,\tilde{q}_{\ell})-{\nabla\,\cdot\,}( {p}_0\tilde{\vw}_{\ell})\right\rangle 
&
\label{EqSlwPrdCfcnt}
\\
&
+\sum\limits_{j=1..\ell-1}\left\langle p_jf_{\ell-j}-{ \nabla\,\cdot\,}(p_j{\vw}_{\ell-j})\right\rangle,
&
\nonumber\\
&
(\pr_t- \Delta)\bar{q}_{\ell}-\bar{\vb}_2\nabla\bar{s}_\ell- \bar{a}_2(\bar{p}_{\ell},\bar{q}_{\ell})=\left\langle a_2(r_{\ell},\tilde{q}_{\ell})+\sum\limits_{j=1..\ell-1}q_jg_{\ell-j}\right\rangle,
&
\label{EqSlwPryCfcnt}\\
&
\bar{s}_{\ell,t}+\nu \bar{s}_{\ell}-{{\kappa}} \bar{q}_{\ell}=\mu_2\Delta \bar{s}_{\ell-1},
&
\label{EqSlwDrvrCfcnt}\\
&
\text{where}\quad \bar{\vv}=\nabla\bar{s}+\bar{\mathrm{V}}(\nabla\bar{s}),\ \vw_{m}=\nabla s_m+{\wtlnbl} s_{m+1}, \quad {r_\ell}=\pcr_\ell-\mathcal{L}^{-1}(\vw_0,\mu_1){\wtlnbl}\astphi\nabla \bar{s}_{\ell}.
&\nonumber
\end{eqnarray}}}
\end{dfntn}
\begin{rmrk}\label{RmLdSlw}
\emph{The right-hand sides in the equations of the slow system of order $\ell$ are $Op_{\ell-1}$-expressions.}
\end{rmrk}
\begin{theorem}\label{ThSnrApprxm}
\emph{\textsf{Let $\ell>0$, and suppose that for every $m=0...\ell-1$, the triad of coefficients $p_{m},q_{m},s_{m}$ in expansion \eqref{AsmpExpnsn} is known and is such that (a) the triad of their mean values, $\bar{p}_{m},\bar{q}_{m},\bar{s}_{m}$, yields a solution to the slow system of order $m$ (where $m=0$ corresponds to the leading slow system); (b) the triad $p_{m-1},q_{m},s_{m}$ satisfies the system of Eqs \eqref{EqPrdCfcnt}--\eqref{EqDrvrCfcnt} with $k=m$. Finally, assume that the slow system of order $\ell$ admits a solution $\bar{p}_{\ell},\bar{q}_{\ell},\bar{s}_{\ell}$. Then the coefficients $p_{\ell},q_{\ell},s_{\ell}$ are determined completely and in such a way that
   \\
   (i) the triad $p_{\ell-1},q_{\ell},s_{\ell}$ solves the system of Eqs \eqref{EqPrdCfcnt}--\eqref{EqDrvrCfcnt} for $k=\ell$, and 
    \begin{eqnarray*}
& 
p_{\ell-1}=\bar{p}_{\ell-1}\astphi+\bar{p}\mathcal{L}^{-1}(\vw_0,\mu_1){ {\wtlnbl}\,\cdot\,}(\astphi\nabla\bar{s}_{\ell-1})+Op_{\ell-2},
\quad 
{q}_{\ell}=\bar{q}_{\ell}+Op_{\ell-2},
\quad
s_{\ell}=\bar{s}_{\ell}+Op_{\ell-2}.
&
 \end{eqnarray*}
    (ii) for $k=\ell+1$, the system of Eqs \eqref{EqPrdCfcnt}--\eqref{EqDrvrCfcnt} admits a solution $p_{\ell},q_{\ell+1},s_{\ell+1}$, such that 
 \begin{eqnarray*}
& 
p_{\ell}=\bar{p}_{\ell}\astphi+\bar{p}\mathcal{L}^{-1}(\vw_0,\mu_1){ {\wtlnbl}\,\cdot\,}(\astphi\nabla\bar{s}_{\ell})+Op_{\ell-1},
\quad 
{q}_{\ell+1}=\bar{q}_{\ell+1}+Op_{\ell-1},
\quad
s_{\ell+1}=\bar{s}_{\ell+1}+Op_{\ell-1},
&
 \end{eqnarray*}
  where the mean values $\bar{q}_{\ell+1},\bar{s}_{\ell+1}$ of coefficients $q_{\ell+1},s_{\ell+1}$ remain undetermined.
  \\
  (iii) for $k=\ell+2$, the Eqs \eqref{EqPrdCfcnt} and \eqref{EqPryCfcnt} allow one to isolate the following unknowns: 
  \begin{eqnarray*}
& 
p_{\ell+1}=\bar{p}_{\ell+1}\astphi+\bar{p}\mathcal{L}^{-1}(\vw_0,\mu_1){ {\wtlnbl}\,\cdot\,}(\astphi\nabla\bar{s}_{\ell+1})+Op_{\ell},
\quad 
{q}_{\ell+2}=\bar{q}_{\ell+2}+Op_{\ell},
 \end{eqnarray*}
 where the mean values $\bar{q}_{\ell+2}$, $\bar{s}_{\ell+1}$, and $\bar{p}_{\ell+1}$ remain undetermined.}} 
\end{theorem}
\noindent\textbf{Proof.}
The compatibility conditions for Eqs \eqref{EqPrdCfcnt}--\eqref{EqDrvrCfcnt}  are equivalent to the orthogonality of their right-hand sides to the kernels of the operators adjoint to those staying in their left-hand sides (see Section~\ref{ScSttng}). This observations leads to the following equalities:
\begin{eqnarray}
& 
\bar{p}_{k-2,t}+{ \nabla\,\cdot\,}\langle p_{k-2}\vw_{0}+p_0\vw_{k-2}\rangle
=\mu_1\Delta \overline{p}_{k-3}+
&
\label{EqSlwPrdCfcnt-0}\\ 
&
\left\langle a_1(p_{k-2},q_{k-2}) + \sum\limits_{j=1..k-3}p_jf_{k-2-j}\right\rangle, \quad \text{where}\quad \vw_{m}=\nabla s_{m}+{\wtlnbl}{s}_{m+1},
&
\nonumber\\
&\pr_\tau \bqx_{k-1}= M^\xi\left(( \Delta-\pr_t)\bqx_{k-2}+a_2(p_{k-2},q_{k-2})+\sum\limits_{j=1..k-3}q_jg_{k-2-j}\right),&
\label{EqSlwPryCfcnt-0}
\\
&
\bar{s}_{k-1,t}+\nu \bar{s}_{k-1}= {{\kappa}} \bar{q}_{k-1}+\mu_2\Delta \bar{s}_{k-2}.
&
\label{EqSlwDrvrCfcnt-0}
\end{eqnarray}
Comparing this system with its predecessors shows that Eq \eqref{EqSlwPrdCfcnt-0} in fact coincides with Eq \eqref{EqSlwPrdCfcnt} of the slow system of order $\ell=k-2$, the compatibility condition for Eq \eqref{EqSlwPryCfcnt-0} coincides with Eq \eqref{EqSlwPryCfcnt} of the slow system of order $\ell=k-2$, and Eq \eqref{EqSlwDrvrCfcnt-0} coincides with Eq \eqref{EqSlwDrvrCfcnt} of order $\ell=k-1$. Hence, resolving the slow system of order $\ell$ enables us to isolate the unknown $q_{\ell+2}$ from Eq \eqref{EqPryCfcnt} with $k=\ell+2$, then $s_{\ell+1}$ from \eqref{EqDrvrCfcnt} with $k=\ell+1$, and finally $p_{\ell+1}$ from 
\eqref{EqPrdCfcnt} with $k=\ell+2$. This gives us 
\begin{eqnarray}
& 
p_{\ell+1}=\pbl_{\ell+1}+\mathcal{L}^{-1}(\vw_0,\mu_1)(\mu_1\Delta p_{\ell-1}-(\pr_t-\mu_1\mathcal{D})p_{\ell}+
&
\nonumber\\
&
\sum\limits_{j=0..\ell}(p_jf_{\ell-j}-{ \nabla\,\cdot\,}(p_j\vw_{\ell-j})   -{ {\wtlnbl}\,\cdot\,}(p_j\vw_{\ell+1-j}))),
&
\label{PL+1=}\\ 
&
\text{where}\quad \vw_{m}=\nabla s_{m}+{\wtlnbl}{s}_{m+1},\ \pbl_{\ell+1}=\bar{p}_{\ell+1}{\Phi}(\nabla\bar{s}+\wtlnbl\tilde{s}_1,\mu_1),
&
\nonumber
\\
&
{q}_{\ell+2}=\bar{q}_{\ell+2}+\pr_\tau^{-1}M^\xi
\left(( \Delta-\pr_t)q_{\ell}+\sum\limits_{j=0..\ell}q_jg_{\ell-j}\right)+
&
\label{QL+2=}
\\
&
{\widetilde{\Delta}}^{-1}\left((\pr_\tau- \mathcal{D}) q_{\ell+1} -( \Delta-\pr_t)q_{\ell}-\sum\limits_{j=0..\ell}q_jg_{\ell-j}\right),
&
\nonumber 
\\
&
s_{\ell+1}=\bar{s}_{\ell+1}+\mathcal{H}^{-1}_{\mu_2}\left({{\kappa}} q_{\ell}-\nu s_{\ell}+\mu_2\Delta s_{\ell-1}-(\pr_{t}-\mu_2\mathcal{D})s_{\ell}\right).
&
\label{SL+1=}
 \end{eqnarray}
 From Eqs \eqref{PL+1=}--\eqref{SL+1=}, it follows that
 \begin{eqnarray}
& 
p_{\ell+1}=\bar{p}_{\ell+1}\astphi+\bar{p}\mathcal{L}^{-1}(\vw_0,\mu_1){ {\wtlnbl}\,\cdot\,}(\astphi\nabla\bar{s}_{\ell+1})+Op_{\ell},
&
\label{PL+1=!}\\ 
&
{q}_{\ell+2}=\bar{q}_{\ell+2}+Op_\ell,
&
\label{QL+2=!}
\\
&
s_{\ell+1}=\bar{s}_{\ell+1}+Op_{\ell-1}.
&
\label{SL+1=!}
 \end{eqnarray}
From equalities \eqref{PL+1=!}--\eqref{SL+1=!} written for the index $\ell$ shifted down by 2, it follows that $p_{\ell-1}$ and $q_{\ell}$ are completely determined under the assumptions of the theorem. Similarly, shifting the same index down by 1 leads to the unknowns $s_{\ell}$ and $p_{\ell}$ being determined completely. Also, $q_{\ell+1}$ is determined up to its mean value. Finally, without shifting the index, we see that $q_{\ell+2}$ and $s_{\ell+1}$ can be determined up to their mean values, and $p_{\ell+1}$ admits the representation \eqref{PL+1=!} (where the first two terms remain undetermined). Hence, all the assertions of  Theorem~\ref{ThSnrApprxm} hold true under the assumptions stated therein.
$\square$ 
\begin{rmrk}\label{RmOnLnrPrcss}
 \emph{For an arbitrarily specified approximation order, the asymptotic solution \eqref{AsmpExpnsn} is attainable by an iterative process consisting of successively solving the slow systems and systems \eqref{EqPrdCfcnt}--\eqref{EqDrvrCfcnt}. These are linear systems, except for the leading slow one. Of course, constructing the whole sequence of approximations includes solving the whole sequence of the slow systems. Exploring their solvability  is beyond the scope of this study. At the same time, this is likely not a very hard problem, as there are significant advances in exploring the solvability of so-called parabolic-hyperbolic models of tactic motions,  which are obviously close relatives to our slow systems; \emph{cf.} articles \cite{FuGrtMgl,FuGrtMgl1} and the references therein. Anyway, we leave this issue for future work.}
 \end{rmrk}
\section{Appendix C. Hankel's determinants.}
%\addcontentsline{toc}{section}{Appendix C. Hankel's determinants.} 
\label{ScApp3}
\noindent
This appendix contains several explicit expressions involving  Hankel's determinants. We assume the parameters to satisfy  restrictions \eqref{ExstCoExst} and \eqref{EqPrmDomn}. For avoiding  too-bulky  expressions, we set $\delta_1=\delta_2=0$, that is, we remove the artificial diffusivity and get back to the eigenvalue problem arising directly from the linearized leading slow system. For the same purpose, we use the following notation:
\beq
a_{12}=a^2,\quad -a_{22}=F^2,\quad a^2 a_{2,1}=-G^2,\quad z=k^2.
\eeq
Also, all the notations of Sections~\ref{SsscNrmlMds} and \ref{SscEffctSgnl} remain valid. {In particular, $p_e>0$ denotes the equilibrium density, which is also the quasi-equilibrium one by the assumed $p$-linearity.}
 Thus, $\Delta_2={F}^{2}+{\nu}+z$,
\beq
\Delta_4=\left( {\nu}^{2}+ \left( {F}^{2}+z \right) \nu+{G}^{2} \right) 
 \left( {F}^{2}+z \right)  \left( {F}^{2}+\nu+z \right)-{{\check{\chi}{G}^{2}p_e\,z\,{\kappa}\, \left( {F}^{2}+{\nu}+z
 \right) }{{a}^{-2}}},
 &\\&
 \chi_{40}(\Pi,k)={\frac { \left( {F}^{2}+z \right)  \left( {{\nu}}^{2}+ \left( {F}^{
2}+z \right) {\nu}+{G}^{2} \right) {a}^{2}}{{G}^{2}p_e\,z\,{\kappa}}},
&\\&
k^2_4(\Pi)=\nu^{-1}FR,\ R=\sqrt{\nu(F^2\nu+G^2+\nu^2)},
&\\&
\chi_4^*(\Pi)=\inf\chi_{40}(\Pi,k)=\frac{{a}^{2} \left({\nu} {F}^{2}+FR+{G}^{2}+{\nu}^{2} \right)(\nu F+R)}{\kappa p_e G^2R}.
\eeq
{Obviously, the  equation $\chi_4^*(\Pi)=0$ allows for isolating the variable $\kappa$  and this determines a smooth function $\kappa=\kappa_4$.} So,  $\mathcal{A}_0\neq\varnothing$. Further,
\beq
\Delta_6=\left({a}^{-2}z{p}_e{\kappa} {G}^{2}\check{\chi}\right)^{3}-
&\\&
2({a}^{-2}\check{\chi}{z}{{\kappa}}{p}_e{G}^{2})^{2}
 \left(  \left( {F}^{2}+z \right) {\nu}^{2}+ \left( {F}^{4}+2\,z{F}^{2
}-{G}^{2}/2+{z}^{2} \right) \nu+{G}^{2} \left( {F}^{2}+z \right) 
 \right) 
&\\&
 +\, {a}^{-2}\check{\chi} z{\kappa}\,p_e({F}^{2}+z ){G}^{2}\left\{ 
 \left( {F}^{2}+z \right) {\nu}^{4}+ \right.
 &\\&
\left. \left( 2\,{z}^{2}+ \left( 4\,{F}^
{2}-{c}^{2} \right) z+2\,{F}^{4}-2\,{G}^{2} \right) {\nu}^{3}+\right.
&\\&
\left. \left( 
{F}^{2}+z \right)\left( {z}^{2}+ \left( 2\,{F}^{2}-2\,{c}^{2}
 \right) z+{F}^{4} \right) {\nu}^{2}+ \right.
 &\\&
 \left.\left( -{c}^{2}{z}^{3}+ \left( -
2\,{F}^{2}{c}^{2}+2\,{G}^{2} \right) {z}^{2}+\left( -{F}^{4}{c}^{2}+4
\,{F}^{2}{G}^{2} \right) z+2\,{G}^{2}{F}^{4}-2\,{G}^{4} \right) \nu+{G
}^{4} \left( {F}^{2}+z \right)  \right\}
&\\&
+\,{\nu}\, \left( {F}^{2}+z \right) ^{2} {G}^{2} \left\{ {{\nu}}^{4}+
 \left( 2\,{F}^{2}+2\,z \right) {{\nu}}^{3}+ \left( {z}^{2}+
 \left( 2\,{F}^{2}+{c}^{2} \right) z+{F}^{4}+2\,{G}^{2} \right) {{\nu}}^{2}+\right.
 &\\&
 \left. 2\, \left( {c}^{2}z+{G}^{2} \right)  \left( {F}^{2}+z
 \right) {\nu}+{F}^{4}{c}^{2}z+2\,{F}^{2}{c}^{2}{z}^{2}+{c}^{2}{z}^
{3}+{G}^{4} \right\}.
\eeq
By this expression, it follows that determinant $\Delta_6$ is a third-order polynomial in $\check{\chi}$, positive for $\check{\chi}=0$ and  $\check{\chi}\to\infty$ (for an arbitrarily fixed collection of the remaining parameters) as we have claimed in Section~\ref{SsscNrmlMds}.
Furthermore, $\Delta_6|_{c=0}\Delta_2^2\Delta_4^{-2}=(\check{\chi} p_e z \kappa+a^2\nu)G^2/a^2$,
as we have claimed in Section~\ref{SsscNrmlMds}.
At the same time, 
\beq
\frac{\pr^5_z \Delta_6}{5!}= {\frac { \left(p_e\,\check{\chi}\, {\kappa}\left({\nu}-{c}^{2} \right)+{c}^{2}{a}^{2} \right) {\nu}\,{G}^{2}}{{a}^{2}}}.
\eeq
Hence, the leading term in $\Delta_6$ remains positive as $k\to+\infty$ provided that $c$ is small enough; say, $c^2<\nu$. This is exactly what we have claimed in Section~\ref{SsscNrmlMds}. Finally, substituting $\check{\chi}={\chi}_{40}+\eps c$ transforms the equation $\Delta_6=0$ to the form $c^2\Delta_6^{\sharp}(c,\eps)=0$, where
\beq
\Delta_6^{\sharp}/z=-{\nu}^{2} \left( {F}^{2}+z \right) ^{2} \left( {F}^{2}+\nu+z \right) 
^{3}{a}^{6}-c \varepsilon\,p_e z\nu\,{\kappa}\, {G}^{2}\left( {F}^{2
}+z \right)  \left( {F}^{2}+\nu+z \right) ^{2}{a}^{4}\,+
&\\&
{\varepsilon}^{2}z{p}^{2}{{\kappa}}^{2}{a}^{2} {G}^{4}\left( {F}^{2}+\nu+z \right) \left(  \left( {F}^{2}+
z \right) \nu+{G}^{2} \right)+c{\varepsilon}^{3}{z}^{2}{p}_e^{3}{{\kappa}}^{3}{G}^{6}.
\eeq
From this expression, it follows that
\beq
\Delta_6^{\sharp}=-\Delta_{602}+\eps^2\Delta_{620}+\ldots,\ \text{where}\ 
&\\&
\Delta_{602}={\nu}^{2}z \left( {F}^{2}+z \right) ^{2} \left( {F}^{2}+\nu+z \right) ^
{3}{a}^{6}>0,
&\\&
\Delta_{620}=z^2\, \left( {F}^{2}+{\nu}+z \right) {G}^{4}{p}^{2} \left(  \left( {F
}^{2}+z \right) {\nu}+{G}^{2} \right) {{\varkappa}}^{2}{a}^{2}.
\eeq
This is exactly what we have claimed in Section~\ref{SsscNrmlMds}.

\begin{thebibliography}{999}\addcontentsline{toc}{bibliography}{REFERENCES}
\bibitem{KrvOdll} P. Kareiva, G. Odell, {Swarms of predators exhibit `preytaxis' if individual predators use area-restricted search}, \textit{The Am. Nat.,} \textbf{130} (1987), 233--270. \url{https://doi.org/10.1086/284707}
\bibitem{Amann} H. Amann, {\em Nonhomogeneous linear and quasilinear elliptic and parabolic boundary value problems}, In: Function spaces, differential operators and nonlinear analysis (pp. 9--126), Springer Fachmedien Wiesbaden, 1993. \url{https://doi.org/10.1007/978-3-663-11336-2}
\bibitem{Ts94} G. R. Ivanitskii, A. B. Medvinskii, M. A. Tsyganov, {From the dynamics of population autowaves generated by living cells to neuroinformatics}, \textit{Phys.-Usp.,} \textbf{37} (1994), 961--989. \url{https://doi.org/10.1070/PU1994v037n10ABEH000049}
\bibitem{Hrstmnn} D. Horstmann, {From 1970 until present: The Keller–Segel model in chemotaxis and its consequences II}, \textit{Jahresber. Deutsch. Math.-Verein.,} \textbf{106} (2004), 51--69.
\bibitem{HllnPntr} T. Hillen, K. J. Painter, {A user's guide to PDE models for chemotaxis}, \textit{J. Math. Biol.,} \textbf{58} (2009), 183--217. \url{https://doi.org/10.1007/s00285-008-0201-3}
\bibitem{BellBellTao} N. Bellomo, A. Bellouquid, Y. Tao, M. Winkler, {Toward a mathematical theory of Keller--Segel models of pattern formation in biological tissues}, \textit{Math. Mod. Meth. Appl. S.,} \textbf{25} (2015), 1663--1763. \url{https://doi.org/10.1142/S021820251550044X}
\bibitem{BrzKrv} F. S. Berezovskaya, G. P. Karev, {Bifurcations of traveling waves in population taxis models}, \textit{Phys.-Usp.,} \textbf{42} (1999), 917--929. \url{https://doi.org/10.1070/PU1999v042n09ABEH000564}
\bibitem{GMT} V. Govorukhin, A. Morgulis, Y. Tyutyunov, {Slow taxis in a predator--prey model}, \textit{Dokl. Math.,} \textbf{61} (2000), 420--422.
\bibitem{AGMTS} R. Arditi, Y. Tyutyunov, A. Morgulis, V. Govorukhin, I. Senina, {Directed movement of predators and the emergence of density-dependence in predator--prey models}, \textit{Theor. Popul. Biol.,} \textbf{59} (2001), 207--221. \url{https://doi.org/10.1006/tpbi.2001.1513}
\bibitem{SapTA} N. Sapoukhina, Y. Tyutyunov, R. Arditi, {The role of prey taxis in biological control: A spatial theoretical model}, \textit{Am. Nat.,} \textbf{162} (2003), 61--76. \url{https://doi.org/10.1086/375297}
\bibitem{Chkr} A. Chakraborty, M. Singh, D. Lucy, P. Ridland, {Predator--prey model with prey-taxis and diffusion}, \textit{Math. Comput. Model.,} \textbf{46} (2007), 482--498. \url{https://doi.org/10.1016/j.mcm.2006.10.010}
\bibitem{TaoW} C. Mu, W. Tao, Z. A. Wang, {Global dynamics and spatiotemporal heterogeneity of a prey-taxis model with prey-induced acceleration}, \textit{Eur. J. Appl. Math.,} 2024, 1--33. \url{https://doi.org/10.1017/S0956792523000347}
\bibitem{LeeHlnLws} J. M. Lee, T. Hillen, M. A. Lewis, {Pattern formation in prey-taxis systems}, \textit{J. Biol. Dynam.,} \textbf{3} (2009), 551--573. \url{https://doi.org/10.1080/17513750802716112}
\bibitem{HllnPntr1}K. J. Painter, T. Hillen, {Spatio-temporal chaos in a chemotaxis model}, \textit{Physica D,} \textbf{240} (2011), 363--375. \url{https://doi.org/10.1016/j.physd.2010.09.011}
\bibitem{BnPtrRtDpnd} M. Banerjee, S. Petrovskii, {Self-organised spatial patterns and chaos in a ratio-dependent predator--prey system}, \textit{Theor. Ecol.,} \textbf{4} (2011), 37--53. \url{https://doi.org/10.1007/s12080-010-0073-1}
\bibitem{Chaplain} I. G. Pearce, M. A. J. Chaplain, P. G. Schofield, A. R. A. Anderson, S. F. Hubbard, {Chemotaxis-induced spatio-temporal heterogeneity in multi-species host-parasitoid systems}, \textit{J. Math. Biol.,} \textbf{55} (2007), 365--388. \url{https://doi.org/10.1007/s00285-007-0088-4}
\bibitem{Grchv} E. Giricheva, {Taxis-driven pattern formation in tri-trophic food chain model with omnivory}, \textit{Mathematics,} \textbf{12} (2024), 290. \url{https://doi.org/10.3390/math12020290}
\bibitem{TllWrzk} I. J. Tello, D. Wrzosek, {Predator--prey model with diffusion and indirect prey-taxis}, \textit{Math. Mod. Meth. Appl. S.,} \textbf{26} (2016), 2129--2162. \url{https://doi.org/10.1142/S0218202516400108}
\bibitem{WngWng} J. Wang, M. Wang, {The dynamics of a predator--prey model with diffusion and indirect prey-taxis}, \textit{J. Dyn. Differ. Equ.,} \textbf{32} (2020), 1291--1310. \url{https://doi.org/10.1007/s10884-019-09778-7}
\bibitem{Li} S. Li, {Positive steady-state solutions for a class of prey-predator systems with indirect prey-taxis}, \textit{SIAM J. Math. Anal.,} \textbf{55} (2023), 6342--6374. \url{https://doi.org/10.1137/22M1529518}
\bibitem{TtnSnTtvBnrj} Y. V. Tyutyunov, D. Sen, L. I. Titova, M. Banerjee, {Predator overcomes the Allee effect due to indirect prey--taxis}, \textit{Ecol. Complex.,} \textbf{39} (2019), 100772. \url{https://doi.org/10.1016/j.ecocom.2019.100772}
\bibitem{WngYngZhng} Q. Wang, J. Yang, L. Zhang, {Time-periodic and stable patterns of a two-competing-species Keller-Segel chemotaxis model: Effect of cellular growth}, \textit{Discrete  Cont. Dyn.-B,} \textbf{22} (2017), 3547--3574. \url{https://doi.org/10.3934/dcdsb.2017179}
\bibitem{WuWnGn} S. Wu, Y. Wang, D. Geng, {Dynamic and pattern formation of a diffusive predator--prey model with indirect prey-taxis and indirect predator-taxis}, \textit{Nonlinear Anal.-Real,} \textbf{84} (2025), 104299. \url{https://doi.org/10.1016/j.nonrwa.2024.104299}
\bibitem{Xie} Z. Xie, Y. Li, {Classical solutions to a pursuit-evasion model with prey-taxis and indirect predator-taxis}, \textit{Math. Meth. Appl. Sci.,} \textbf{48} (2025), 4117--4143. \url{https://doi.org/10.1002/mma.10536}
\bibitem{TtnZgr} Y. Tyutyunov, L. Titova, I. Senina, {Prey-taxis destabilizes homogeneous stationary state in spatial Gause--Kolmogorov-type model for predator--prey system}, \textit{Ecol. Complex.,} \textbf{31} (2017), 170--180. \url{https://doi.org/10.1016/j.ecocom.2017.07.001}
\bibitem{Southall} B. L. Southall, J. J. Finneran, C. Reichmuth, P. E. Nachtigall, D. R. Ketten, A. E. Bowles, et al., {Marine mammal noise exposure criteria: Updated scientific recommendations for residual hearing effects}, \textit{Aquat. Mamm.,} \textbf{45} (2019), 125--232. \url{https://doi.org/10.1578/AM.45.2.2019.125}
\bibitem{Hsu} A. C. Hsu, A. M. Boustany, J. J. Roberts, J. Chang, P. N. Halpin, {Tuna and swordfish catch in the U.S. northwest Atlantic longline fishery in relation to mesoscale eddies}, \textit{Fish. Oceanogr.,} \textbf{24} (2015), 508--520. \url{https://doi.org/10.1111/fog.12125}
\bibitem{Royer} F. Royer, J. M. Fromentin, H. Farrugio, P. Gaspar, {Determining bluefin tuna habitat through frontal features in the Mediterranean Sea}, \textit{Collect. Vol. Sci. Pap.,} \textbf{58} (2005), 1275--1284. Available from: \\ \url{https://archimer.ifremer.fr/doc/00194/30504}.
\bibitem{Reese} D. C. Reese, R. T. O'Malley, R. D. Brodeur, J. H. Churnside, {Epipelagic fish distributions in relation to thermal fronts in a coastal upwelling system using high-resolution remote-sensing techniques}, \textit{ICES J. Mar. Sci.,} \textbf{68} (2011), 1865--1874. \url{https://doi.org/10.1093/icesjms/fsr107}
\bibitem{Kratina} P. Kratina, R. M. Lecraw, T. Ingram, B. R. Anholt, {Stability and persistence of food webs with omnivory: Is there a general pattern?}, \textit{Ecosphere,} \textbf{3} (2012), 1--18. \url{https://doi.org/10.1890/es12-00121.1}
\bibitem{Chdh1} P. R. Chowdhury, S. Petrovskii, M. Banerjee, {Oscillations and pattern formation in a slow--fast~prey--predator system}, \textit{B. Math. Biol.,} \textbf{83} (2021), 1--41. \url{https://doi.org/10.1007/s11538-021-00941-0}
\bibitem{Chdh2} P. R. Chowdhury, S. Petrovskii, M. Banerjee, {Effect of slow--fast time scale on transient dynamics in a realistic prey-predator system}, \textit{Mathematics,} \textbf{10} (2022), 699. \url{https://doi.org/10.3390/math10050699}
\bibitem{Chdh3} P. Roy Chowdhury, M. Banerjee, S. Petrovskii, {A two-timescale model of plankton--oxygen dynamics predicts formation of oxygen minimum zones and global anoxia}, \textit{J. Math. Biol.,} \textbf{89} (2024), 8. \url{https://doi.org/10.1007/s00285-024-02107-7}
\bibitem{PgGd} J. C. Poggiale, M. Gauduchon, {A general class of growth models in a patchy environment exhibiting enhanced production}, \textit{Mathematics,} \textbf{14} (2026), 1165. \url{https://doi.org/10.3390/math14071165}
%\bibitem{R3-1}[10.1140/epjp/s13360-025-07123-0] N. Sk, B. Mondal, M. Ch-Chaoui, R. K. Upadhyay, S. Mandal, P. K. Tiwari, {Influence of Allee effect and cooperative hunting on the dynamics of a modified Hastings--Powell model in deterministic, discrete, and stochastic environments}, \textit{Eur. Phys. J. Plus,} \textbf{140} (2025), 1178. \url{https://doi.org/10.1140/epjp/s13360-025-07123-0}

%\bibitem{R3-2}[10.1016/j.matcom.2024.10.017] K. Mokni, H. Ben Ali, B. Ghosh, M. Ch-Chaoui, {Nonlinear dynamics of a Darwinian Ricker system with strong Allee effect and immigration}, \textit{Math. Comput. Simul.,} \textbf{229} (2025), 789--813. \url{https://doi.org/10.1016/j.matcom.2024.10.017}
\bibitem{TllWrzLst} J. Ignacio Tello, D. Wrzosek, {From indirect to direct taxis by fast reaction limit}, \textit{Math. Mod. Meth. Appl. S.,} \textbf{36} (2026), 173--204. \url{https://doi.org/10.1142/S0218202526500041}
\bibitem{AM1} A. Morgulis, K. Ilin, {Indirect taxis on a fluctuating environment}, \textit{Mathematics,} \textbf{8} (2020), 2052. \url{https://doi.org/10.3390/math8112052}
\bibitem{AM3} A. Morgulis, K. H. Malal, {Prey-taxis vs. an external signal: Short-wave asymptotic and stability analysis}, \textit{Mathematics,} \textbf{13} (2025), 261. \url{https://doi.org/10.3390/math13020261}
\bibitem{AM4} A. Morgulis, K. Malal, {Prey-taxis vs a shortwave external signal in multiple dimensions}, \textit{J. Math. Sci.,} 2025. \url{https://doi.org/10.1007/s10958-025-08063-x}
\bibitem{LndLfs} L. D. Landau, E. M. Lifshitz, {\em Mechanics}, 3 Eds., Elsevier, 1982, 224.
\bibitem{Yudovich3} V. I. Yudovich, {The dynamics of a particle on a smooth vibrating surface}, \textit{J. Appl. Math.  Mech.,} \textbf{62} (1998), 893--900. \url{https://doi.org/10.1016/S0021-8928(98)00114-2}
\bibitem{Yudovich4} V. I. Yudovich, {Vibration dynamics of systems with constraints}, \textit{Dokl. Math.,} \textbf{42} (1997), 322--325.
\bibitem{ZenSimIzvAN66} S. M. Zen'kovskaya, I. B. Simonenko, {Effect of high frequency vibration on convection initiation}, \textit{Fluid Dynam.,} \textbf{1} (1966), 35--37.
\bibitem{ZenMZhG68} S. M. Zen'kovskaya, {Study of convection in a liquid layer with vibrating forces}, \textit{Fluid Dynam.,} \textbf{3} (1968), 35--36.
\bibitem{MrVld} V. A. Vladimirov, A. B. Morgulis, {Relative equilibria in the Bjerknes problem}, \textit{Siberian Math. J.,} \textbf{55} (2014), 35--48. \url{https://doi.org/10.1134/S0037446614010054}
\bibitem{LbmvBook} D. V. Lyubimov, T. P. Lyubimova, A. A. Tcherepanov, B. H. Roux, {Vibration influence on fluid interfaces}, \textit{Comptes Rendus. Mécanique,} \textbf{332} (2004), 467--472. \url{https://doi.org/10.1016/j.crme.2004.01.013}
\bibitem{Tvstk} A. K. Belyaev, N. F. Morozov, P. E. Tovstik,  T. P. Tovstik, {The stability of a flexible vertical rod on a vibrating support}, \textit{Vestnik St. Petersb. Univ. Math.,} \textbf{51} (2018), 296--304. \url{https://doi.org/10.3103/S1063454118030020}
\bibitem{TvstkT} A. V. Babenko, O. R. Polyakova, T. P. Tovstik, {\em Conceptual generalizations of the Kapitsa problem},~In: Altenbach, H., Irschik, H., Porubov, A.V. (eds) Progress in Continuum Mechanics, Advanced Structured Materials, vol 196, Cham: Springer, 2023. \url{https://doi.org/10.1007/978-3-031-43736-6-4}
\bibitem{Allr-1} G. Allaire, {A brief introduction to homogenization and miscellaneous applications}, \textit{ESAIM: Proceedings,} \textbf{37} (2012), 1--49. \url{https://doi.org/10.1051/proc/201237001}
\bibitem{TtnTrFn}Yu. V. Tyutyunov, L. I. Titova, {Ratio-dependence in predator--prey systems as an edge and basic minimal model of predator interference}, \textit{Front. Ecol. Evol.,} \textbf{9} (2021), 725041. \url{https://doi.org/10.3389/fevo.2021.725041}   
\bibitem{ArdG} R. Arditi, L. R. Ginzburg, {\em How species interact: Altering the standard view on trophic ecology}, Oxford: Oxford University Press, 2012.
\bibitem{Nrnbrg} L. Nirenberg, {A strong maximum principle for parabolic equations}, \textit{Commun. Pur. Appl. Math.,} \textbf{6} (1953), 167--177.
\bibitem{Lnds} E. M. Landis, {\em Second order equations of elliptic and parabolic type}, Providence: American Math. Soc., 1997, 203.
% \bibitem{Ydvch} V. I. Yudovich, {\em The linearization method in hydrodynamical stability theory}, Providence: American Mathematical Soc., 1989, 170 p.
% \bibitem{Iss} M. Haragus, G. Iooss, {\em Local bifurcations, center manifolds, and normal forms in infinite-dimensional dynamical systems}, Berlin: Springer Science \& Business Media, 2010.
\bibitem{GntMtrx} F. R. Gantmacher, {\em Theory of matrices, Vol.~2}, New York: Chelsea Publishing, 1959.
%\bibitem{Ti1}[10.1007/s00208-023-02698-4] A. Chatziafratis, T. Ozawa, S. F. Tian, {Rigorous analysis of the unified transform method and long-range instabilities for the inhomogeneous time-dependent Schrödinger equation on the quarter-plane}, \textit{Math. Ann.,} \textbf{389} (2024), 3535--3602. \url{https://doi.org/10.1007/s00208-023-02698-4}

%\bibitem{Ti2}[10.1016/j.aim.2025.110552] J. J. Yang, S. F. Tian, Z. Q. Li, {The modified Camassa--Holm equation with nonzero background: Soliton resolution conjecture and asymptotic stability of N-soliton solutions}, \textit{Adv. Math.,} \textbf{484} (2026), 110552. \url{https://doi.org/10.1016/j.aim.2025.110552}

%\bibitem{Ti3}[10.1007/s00208-026-03316-9] T. Ozawa, S. F. Tian, Z. J. Wu, {Orbital stability of dark solitons for the modified Korteweg-de Vries hierarchy with nonzero boundary conditions}, \textit{Math. Ann.,} \textbf{394} (2026), 23. \url{https://doi.org/10.1007/s00208-026-03316-9}
\bibitem{LevCnv1993} V. B. Levenshtam, {Justification of the averaging method for the convection problem with high-frequency vibrations}, \textit{Siberian Math. J.,} \textbf{34} (1993), 280--296. \url{https://doi.org/10.1007/BF00970954}
\bibitem{LevIzvRAN06} V. B. Levenshtam, {Justification of the averaging method for parabolic equations containing rapidly oscillating terms with large amplitudes}, \textit{Izv. Math.,} \textbf{70} (2006), 233--263. \url{https://doi.org/10.1070/IM2006v070n02ABEH002311}
\bibitem{LevAA14} V. B. Levenshtam, {Justification of the averaging method for a system of equations with the Navier--Stokes operator in the principal part}, \textit{St. Petersb. Math. J.,} \textbf{26} (2015), 69--90. \url{https://doi.org/10.1090/S1061-0022-2014-01331-3}
\bibitem{LevSmz23} V. B. Levenshtam, {Averaging a high-frequency hyperbolic system of quasilinear equations with large summands}, \textit{Siberian Math. J.,} \textbf{65} (2024), 934--942. \url{https://doi.org/10.1134/S0037446624040189}
\bibitem{Allr-2} G. Allaire, {Homogenization and two-scale convergence}, \textit{SIAM J. Math. Anal.,} \textbf{23} (1992), 1482--1518. \url{https://doi.org/10.1137/0523084}
\bibitem{MrchKhrs} V. A. Marchenko, E. Y. Khruslov, {\em Homogenization of partial differential equations}, Birkhauser, 2006. \url{https://doi.org/10.1007/978-0-8176-4468-0}
\bibitem{FuGrtMgl} X. Fu, Q. Griette, P. Magal, {Sharp discontinuous traveling waves in a hyperbolic Keller--Segel equation}, \textit{Math. Mod.  Meth. Appl. S.,} \textbf{31} (2021), 861--905. \url{https://doi.org/10.1142/S0218202521500214}
\bibitem{FuGrtMgl1} X. Fu, Q. Griette, P. Magal, {A cell--cell repulsion model on a hyperbolic Keller--Segel equation}, \textit{J. Math. Biol.,} \textbf{80} (2020), 2257--2300. \url{https://doi.org/10.1007/s00285-020-01495-w}
\end{thebibliography}
\end{document}